\documentclass[a4paper,12pt]{amsart}
\usepackage{geometry}

\usepackage{amsmath,amsthm,amsfonts,amssymb,amscd}
\usepackage{fullpage}
\usepackage{lastpage}
\usepackage{enumerate}
\usepackage{fancyhdr}

\usepackage[american]{babel}

\usepackage{mathrsfs}
\usepackage{latexsym,epsfig,hyperref}          

\usepackage{pstricks}
\usepackage{enumitem}
\usepackage{palatino}

\usepackage{tikz-cd}

\newtheorem{thm}{Theorem}
\newtheorem{remark}[thm]{Remark}
\newtheorem{prop}[thm]{Proposition}
\newtheorem{proposition}[thm]{Proposition}

\newtheorem{lemma}[thm]{Lemma}

\newtheorem*{theoremA}{Theorem A}
\newtheorem*{theoremB}{Theorem B}

\def\dim{\operatorname{dim}}

\def\Pic{\operatorname{Pic}}%
\def\Aut{\operatorname{Aut}}%
\def\gon{\operatorname{gon}}%
\def\Div{\operatorname{Div}}%
\def\elm{\operatorname{elm}}%
\def\Bl{\operatorname{Bl}}%
\def\supp{\operatorname{supp}}%

\title{Components of the Hilbert Scheme of smooth projective curves using ruled surfaces II: existence of non-reduced components}

\author[Y. Choi]{Youngook Choi}
\address{Department of Mathematics Education, Yeungnam University, 280 Daehak-Ro, \hfill \newline\texttt{}
 \indent Gyeongsan, Gyeongbuk 38541, Republic of Korea}
\email{ychoi824@yu.ac.kr}

\author[H. Iliev]{Hristo Iliev}
\address{American University in Bulgaria, 2700 Blagoevgrad, Bulgaria, and \hfill \newline\texttt{}
 \indent Institute of Mathematics and Informatics, Bulgarian Academy of Sciences, \hfill \newline\texttt{}
 \indent 1113 Sofia, Bulgaria}
\email{ hiliev@aubg.edu, hki@math.bas.bg}

\author[S. Kim]{Seonja Kim}
\address{Department of Electronic Engineering, Chungwoon University, Sukgol-ro, Michuhol-gu, \hfill \newline\texttt{}
 \indent Incheon 22100, Republic of Korea}
\email{sjkim@chungwoon.ac.kr}

\thanks{The first author was supported by Basic Science Research Program through the National Research Foundation of Korea(NRF)
funded by the Ministry of Education(2019R1I1A3A01055643).
The second author was supported by Grant KP-06-N 62/5 of Bulgarian National Science Fund.
The third  author was supported by by the National Research Foundation of Korea(NRF) grant funded by the Korea government(MSIT) (2022R1A2C1005977).}

\subjclass[2000]{Primary 14C05; Secondary 14H10}
\keywords{Hilbert scheme of curves, double covering, ruled surfaces}

\usepackage{setspace}
\begin{document}

\setlength{\parindent}{5ex}

\begin{abstract}
For $\gamma \geq 7$ and $g \geq 6\gamma + 5$, we construct a family $\mathcal{F}^{\prime}$ of curves lying on cones in $\mathbb{P}^{g-3\gamma+1}$ over smooth non-degenerate curves of genus $\gamma$ and degree $g-2\gamma$ in $\mathbb{P}^{g-3\gamma+1}$. We show that $\dim \mathcal{F}^{\prime} = 2g-\gamma-1 + (g-3\gamma+1)^2$. For a general curve $X^{\prime}$ from the family $\mathcal{F}^{\prime}$, we compute the dimension of the space of its first-order deformations. We prove that the family $\mathcal{F}^{\prime}$ gives rise to an irreducible, non-reduced component $\mathcal{D}^{\prime}$ of the Hilbert scheme $\mathcal{I}_{2g-4\gamma + 1, g, g - 3\gamma + 1}$, which parametrizes smooth, irreducible, non-degenerate curves of degree $2g-4\gamma + 1$ and genus $g$ in $\mathbb{P}^{g-3\gamma+1}$. We obtain $\dim T_{[X^{\prime}]} \mathcal{D}^{\prime} = \dim \mathcal{D}^{\prime} + 1 = \dim \mathcal{F}^{\prime} + 1$.
\end{abstract}

\subjclass[2000]{Primary 14C05; Secondary 14H10}
\keywords{Hilbert scheme of curves, Brill-Noether theory, double covering}

\maketitle

\section{Introduction}\label{Section_1}

Denote by $\mathcal{I}_{d,g,r}$ the union of irreducible components of the Hilbert scheme whose general points correspond to smooth, irreducible, non-degenerate curves of degree $d$ and genus $g$ in $\mathbb{P}^r$. The geometry of $\mathcal{I}_{d,g,r}$ can be quite complex, as it may have multiple components of different dimension that might further be generically smooth or non-reduced. Recall that the minimal dimension of a component is given by $\lambda_{d,g,r} := (r+1)d - (r-3)(g-1)$, known as the \emph{expected dimension} of a component of $\mathcal{I}_{d,g,r}$.

When the \emph{Brill-Noether number} $\rho(d,g,r) := g-(r+1)(g-d+r)$ is non-negative, the scheme $\mathcal{I}_{d,g,r}$ possesses the unique component dominating the moduli space of curves $\mathcal{M}_g$, known as the \emph{distinguished component}. It is generically smooth and has the expected dimension, \cite{Ser84}. When $\rho(d,g,r) > 0$ and $d \geq \eta_3 := 2(1-\frac{3}{r+1})g + 1 + \frac{12}{r+1}$, the scheme $\mathcal{I}_{d,g,r}$ is irreducible (see \cite[Theorem 3.3]{Kim2001} for details).

If we fix $g$ and $r$ and trace the change of geometry of $\mathcal{I}_{d,g,r}$ as the degree $d$ decreases from $\eta_3$, we observe interesting phenomena. For $\eta_4 := 2(1-\frac{4}{r+1})g + 1 + \frac{26}{r+1} < d \leq \eta_3$, the scheme $\mathcal{I}_{d,g,r}$ is irreducible if and only if $3 \nmid (2g-2 - d)$. For $d \leq \eta_4$, the scheme $\mathcal{I}_{d,g,r}$ is reducible. In particular, for $d \leq 2(1-\frac{8}{r+3})g + 2 + \frac{8}{r+3} =: \xi_2$, it acquires components that parametrize curves that are double covers (see \cite[Proposition 3.2, Theorem 4.3 and Theorem 4.4]{CIK17} and \cite[Theorem A]{CIK21}).

Our paper \cite{CIK21} extended the result \cite[Theorem 4.3]{CIK17} by showing that the corresponding components are generically smooth. We recall its main result here for the convenience of the reader.

\begin{theoremA}[\cite{CIK21}]
Assume that $g$ and $\gamma$ are integers with $g \geq 4\gamma -2 \geq 38$. Let
\[
  d:= 2g - 4\gamma + 2 \quad \mbox{ and } \quad \max \left\lbrace \gamma, \frac{2(g-1)}{\gamma} \right\rbrace \leq r \leq R:= g
- 3\gamma + 2 \, .
\]
Then the Hilbert scheme $\mathcal{I}_{d, g, r}$ possesses a generically reduced component $\mathcal{D}_{d, g, r}$ for which
\[
  \dim \mathcal{D}_{d, g, r} = \lambda_{d, g, r} + r\gamma - 2g+2 \, .
\]
Further, let $X_r \subset \mathbb{P}^r$ be a smooth curve corresponding to a general point  of $\mathcal{D}_{d, g, r}$.
\begin{enumerate}[label=(\roman*), leftmargin=*, font=\rmfamily]
  \item If $r = R$ then $X_R$ is the intersection of a general quadric  hypersurface with a cone over a smooth curve $Y$ of degree $g - 2\gamma + 1$ and genus $\gamma$ in $\mathbb{P}^{R-1}$ and $X_R$ is embedded in $\mathbb{P}^R$ by the complete linear series $|R_{\varphi}|$ on $X_R$, where $R_{\varphi}$ is the ramification divisor of the natural projection morphism $\varphi : X_R \to Y$ of degree 2 given by the ruling of cone;

  \item If $r < R$ then $X_r$ is given by a general projection of some $X_R$ as in {\rm (i)}, that is, $X_r$ is embedded in $\mathbb{P}^{r}$ by a general linear subseries $g^r_d$ of $|R_{\varphi}|$.
\end{enumerate}
\end{theoremA}

The main result of the present work establishes the existence of families of curves related to those from \cite[Theorem 4.4]{CIK17}, which give rise to non-reduced components of the Hilbert scheme, as contained in the next theorem.

\begin{theoremB}
Assume that $g$ and $\gamma$ are integers such that $\gamma \geq 7$ and $g \geq 6\gamma + 5$. Let
\[
  d:= 2g - 4\gamma + 2 \quad \mbox{ and } \quad R:= g - 3\gamma + 2 \, .
\]
Then the Hilbert scheme $\mathcal{I}_{d-1, g, R-1}$ possesses a \emph{non-reduced} component $\mathcal{D}^{\prime}_{d-1, g, R-1}$ such that:
\begin{enumerate}[label=(\roman*), leftmargin=*, font=\rmfamily]
 \item $\dim \mathcal{D}^{\prime}_{d-1, g, R-1} = \lambda_{d-1, g, R-1} + (R-1)\gamma - 2g+3 =  2g - \gamma - 1 +R^2 $;

 \item at a general point $[X^{\prime}] \in \mathcal{D}^{\prime}_{d-1, g, R-1}$ we have $\dim T_{[X^{\prime}]} \mathcal{D}^{\prime}_{d-1, g, R-1} = \dim \mathcal{D}^{\prime}_{d-1, g, R-1} + 1$;

 \item a general point $[X^{\prime}] \in \mathcal{D}^{\prime}_{d-1, g, R-1}$ represents a curve $X^{\prime}$ lying on a cone $F^{\prime}$ over a smooth curve $Y^{\prime}$ of genus $\gamma$ and degree $g - 2\gamma $ in $\mathbb{P}^{R-2}$ such that

 \begin{enumerate}[label=\theenumi.\arabic*.]
  \item $X^{\prime} \subset \mathbb{P}^{R-1}$ is linearly normal and passes through the vertex $P^{\prime}$ of the cone $F^{\prime}$;

  \item the projection from $P^{\prime}$ to the hyperplane in $\mathbb{P}^{R-1}$ containing $Y^{\prime}$ induces a morphism ${\varphi}^{\prime}  : X^{\prime} \to Y^{\prime}$ of degree two ;

  \item $\arrowvert \mathcal{O}_{X^{\prime}} (1)\arrowvert \equiv \arrowvert R_{{\varphi}^{\prime}} - \zeta \arrowvert$ on $X^{\prime}$, where $R_{{\varphi}^{\prime}}$ is the ramification divisor of ${\varphi}^{\prime}$ and $\zeta \in X^{\prime}$ is the second point, apart from $P^{\prime}$, in which the tangent line to $X^{\prime}$ at $P^{\prime}$ meets $X^{\prime}$.
 \end{enumerate}
\end{enumerate}
\end{theoremB}

We remark that David Mumford showed in his famous paper on the pathologies of Hilbert scheme \cite{Mum62} that $\mathcal{I}_{14,24,3}$ is non-reduced. Since then, examples of non-reduced components of the Hilbert scheme of space curves $\mathcal{I}_{d, g, 3}$ have been produced in \cite{GP82, Kle87, Ell87, MP96, Nas06, KO15, Dan17}. However, all of them have relied on allocating a family of curves with desired invariants on a surface in $\mathbb{P}^3$. This suggests that to construct non-reduced components of $\mathcal{I}_{d, g, r}$ for $r \geq 4$, a different approach is needed.

Note that showing that a component of $\mathcal{I}_{d,g,r}$ is non-reduced is equivalent to showing that a curve  corresponding to a general point of the component has a first order deformation that does not extend infinitesimally. This is more difficult to do when $r \geq 4$. To our knowledge, the only example of a non-reduced component of $\mathcal{I}_{d,g,r}$ for $r \geq 4$ is the one given by Ciliberto, Lopez and Miranda (see \cite[Theorem 4.11]{CLM96}).

The primary objective of our study was to develop our understanding of the geometry of components of $\mathcal{I}_{d,g,r}$, particularly the families of curves related to those in \cite[Theorem 4.4]{CIK17}. This work builds upon our previous papers \cite{CIK17} and \cite{CIK21} and should be regarded in their context. For this reason, we fix the numbers $d$ and $R$ throughout the paper as
\[
  d:= 2g - 4\gamma + 2 \quad \mbox{ and } \quad R:= g - 3\gamma + 2 \, ,
\]
exactly as they appeared in \cite[Theorem A]{CIK21}, where we showed the existence of a reduced irreducible component $\mathcal{D}_{d,g,R}$ of $\mathcal{I}_{d,g,R}$ parametrizing curves on cones over smooth curves of genus $\gamma$ and degree $g - 2\gamma + 1$ in $\mathbb{P}^{R-1}$. The curves parametrized by $\mathcal{D}^{\prime}_{d-1,g,R-1}$ in {\rm Theorem B} arise as inner projections  of the former (hence the 'prime' in the notation). This allows us to allocate them on cones and compute both the dimension of the family and the dimension of the tangent space  to $\mathcal{I}_{d-1,g,R-1}$ at a general $[X^{\prime}] \in \mathcal{D}^{\prime}_{d-1,g,R-1}$. In Section 2, we provide a geometric explanation of how the family of curves parametrized by $\mathcal{D}^{\prime}_{d-1, g, R-1}$ was found, and we also prove an auxiliary result. The rigorous algebraic construction of the family itself and the proof of {\rm Theorem B} are presented in Section 3. In the final section, we offer further details on the relationship between the families of curves parametrized by the components from Theorem A and Theorem B.

The proof of {\rm Theorem B} proceeds in three steps: constructing the family $\mathcal{F}^{\prime}$ of smooth curves with the desired properties; computing the dimension of the tangent space to $\mathcal{D}^{\prime}_{d-1, g, R-1}$ at a general point $[X^{\prime}] \in \mathcal{D}^{\prime}_{d-1, g, R-1}$, where we take $\mathcal{D}^{\prime}_{d-1, g, R-1}$ to be the closure in $\mathcal{I}_{d-1, g, R-1}$ of the set parametrizing $\mathcal{F}^{\prime}$; and showing that $\mathcal{D}^{\prime}_{d-1, g, R-1}$ is not properly contained in any component of $\mathcal{I}_{d-1, g, R-1}$.

We consider schemes over $\mathbb{C}$. By a \emph{curve}, we understand a smooth, integral, projective curve. Given a line bundle $L$ on a smooth projective variety $X$, we denote by $\arrowvert L \arrowvert$ the complete  linear series  $\mathbb P\left(H^0(X,L)\right)$ on $X$. When $X$ is an object of a family, we denote by  $[X]$ the corresponding point of the Hilbert scheme representing the family.  Often, we abbreviate the notation of a line bundle $L \otimes \mathcal{O}_X (D)$ to $L(D)$ for a divisor $D$ on $X$.  We use $\sim$ to denote linear equivalence of divisors. For a smooth curve $C$ of genus $g$, we denote by $\gon (C)$ the gonality of $C$, which is the minimal $k$ for which there exists a pencil $g^1_k$ on $C$. Recall that if $[C] \in \mathcal{M}_g$ is general, then $\gon (C) = \lfloor \frac{g+3}{2} \rfloor$. For all definitions and properties of the objects not explicitly introduced in the paper, refer to \cite{Hart77} and \cite{ACGH}.

Since this work is closely related to the results obtained in our two previous papers \cite{CIK17} and \cite{CIK21}, and as we refer to them often in what follows, we have opted to fix the values of $d$ as $d = 2g - 4\gamma + 2$ and $R$ as $R = g - 3\gamma + 2$, and to denote the component in {\rm Theorem B} as $\mathcal{D}^{\prime}_{d-1, g, R-1}$. We believe that this allows the results from the three papers to be viewed and understood in the same context.

\section{Motivation and preliminary results}\label{Section_2}

In \cite[Theorem 4.3]{CIK17}, we established the existence of a component $\mathcal{D}_{d, g, r}$ of $\mathcal{I}_{d, g, r}$ under specific numerical assumptions regarding  $d$, $g$ and $r \leq R$. It was shown that a general point of $\mathcal{D}_{d, g, r}$ corresponds to a curve $X$, which is a double covering $\varphi : X \to Y$ of a curve $Y$ of genus $\gamma$, where $\gamma = \frac{2g+2 - d}{4}$. The embedding $X \hookrightarrow \mathbb{P}^r$ was given by a sub-series of the space of global sections of $\mathcal{O}_X (R_{\varphi}) = \omega_X \otimes (\varphi^{\ast} \omega_Y)^{\vee}$, where $R_{\varphi}$ is the ramification divisor of $\varphi$. Similarly, \cite[Theorem 4.4]{CIK17} showed the existence of a component $\mathcal{D}^{\prime}_{d-1, g, r-1}$ of $\mathcal{I}_{d-1, g, r-1}$ whose general point also represents a double covering $\varphi^{\prime} : X^{\prime} \to Y^{\prime}$ of a curve of genus $\gamma$. However, the embedding in this case was given by a sub-series of $|\mathcal{O}_{X^{\prime}} (R_{\varphi^{\prime}}) (-x)|$, determined by the ramification divisor $R_{\varphi^{\prime}}$ and an arbitrary point $x \in X^{\prime}$.

To further explore the geometry of the family of curves parametrized by $\mathcal{D}_{d, g, R} \subset \mathcal{I}_{d, g, R}$, we reconstructed it in \cite[Theorem A]{CIK21} as quadratic hypersurface sections of cones in $\mathbb{P}^R$ over smooth curves of genus $\gamma$ and degree $g - 2\gamma + 1$ in $\mathbb{P}^{R-1}$. We remark that the curves parametrized by $\mathcal{D}_{d, g, r}$ for $r < R$ were obtained as generic projections to a general $\mathbb{P}^r \subset \mathbb{P}^R$. In the discussion that follows, we focus on the case $r = R$.

Since both coverings $\varphi : X \to Y$ and $\varphi^{\prime} : X^{\prime} \to Y^{\prime}$ were general, we can identify the curves parametrized by $\mathcal{D}^{\prime}_{d-1, g, R-1}$ as inner projections of the curves parametrized by $\mathcal{D}_{d, g, R}$. Specifically, let $X \hookrightarrow \mathbb{P}^R$ correspond to a general point of $\mathcal{D}_{d, g, R}$. According to  \cite[Theorem A]{CIK21}, $X$ lies on a cone $F \subset \mathbb{P}^R$ over a general curve $Y$ of genus $\gamma$ and degree $g - 2\gamma + 1$ in $\mathbb{P}^{R-1}$, and $X$ is cut on $F$ by a general quadratic hypersurface in $\mathbb{P}^R$. Let $P$ denote the vertex of $F$. The projection of $X$ from $P$ to the hyperplane $\mathbb{P}^{R-1}$ containing $Y$, or equivalently the ruling of $F$, presents $X$ as a double covering $\varphi : X \to Y$. According to \cite{CIK21}, $\mathcal{O}_X (1) = \mathcal{O}_X (R_{\varphi}) = \omega_X \otimes (\varphi^{\ast} \omega_Y)^{\vee}$.

\begin{prop}\label{Sec2:PropInnerProj}
Using the setup in the preceding paragraph, let $x \in X$ be a point, and let $H \subset \mathbb{P}^R$ be a general hyperplane. Assume that $Y$ contains $x$. Consider the linear projection
\[
\pi_x \, : \, F \dashrightarrow H
\]
with center $x$, see Figure \ref{Figure_1}. Then
\begin{enumerate}[label=(\arabic*), leftmargin=*, font=\rmfamily]
    \item The restriction ${\pi_x}_{|_X} \, : \, X \dashrightarrow H$ extends to an embedding $\widetilde{{\pi_x}_{|_X}} : X \hookrightarrow H \cong \mathbb{P}^{R-1}$ as the degree of $X^{\prime} := \widetilde{{\pi_x}_{|_X}}(X)$ is $\deg X^{\prime} = \deg X - 1 = d - 1 = 2g - 4\gamma + 1$.

    \item the restriction ${\pi_x}_{|_Y} \, : \, Y \dashrightarrow H$ extends to an embedding $\widetilde{{\pi_x}_{|_Y}} : Y  \hookrightarrow H \cap H_Y \cong \mathbb{P}^{R-2}$, where $H_Y$ is the hyperplane of $\mathbb{P}^{R}$ that cuts $Y$ on $F$ as the degree of $Y^{\prime} := \widetilde{{\pi_x}_{|_Y}} (Y)$ is $\deg Y^{\prime} = \deg Y -1 = g -2\gamma$;

    \item The rational map $\pi_x \, : \, F \dashrightarrow H$
    gives rise to a transformation $\widetilde{\pi_x}$ of $F$ into $F^{\prime}$ such that
    \begin{itemize}
        \item $\widetilde{\pi_x} (z) = \pi_x(z)$ for any point $z \in F \setminus \{x\}$.

        \item $\widetilde{\pi_x}$ transforms the vertex $P$ into the vertex $P^{\prime}$ of $F^{\prime}$.

        \item $\widetilde{\pi_x}$ transforms the point $x$ into the line from the ruling of the cone $F^{\prime}$ determined by $P^{\prime}$ and $\widetilde{{\pi_x}_{|_Y}}(x)$.
    \end{itemize}
\end{enumerate}
\end{prop}

\begin{figure}[h]
  \centering
\begin{tikzpicture}[scale=0.62]
    \draw[line width=1pt, gray!50] plot coordinates{(-7,1) (-4,7)};
    \draw[line width=1pt, gray!50] plot coordinates{(-3,0) (-6,7.5)};
    \draw [dashed, line width=0.5pt, gray!50] plot coordinates{(-7, 1) (-7.5,0)};
    \draw [dashed, line width=0.5pt, gray!50] plot coordinates{(-3, 0) (-2.75,-0.5)};
    \draw [line width=1.5pt, red] plot [smooth, tension=0.5] coordinates{(-7,1) (-6.5,0.5) (-6,0.1) (-5,-0.2) (-4,-0.15) (-3,0)};
    \draw [dashed, line width=1.5pt, red] plot [smooth, tension=0.5] coordinates{(-7,1) (-6.5,1.2) (-6,1.25) (-5,1.2) (-4,0.8) (-3,0)};
    \node[red] at (-3.2, -0.35) {{\tiny $Y$}};
    \draw[line width=1pt, gray!50] plot [smooth, tension=0.5] coordinates{(-6,7.5) (-5.5, 7.15) (-5, 6.95) (-4.5, 6.9) (-4,7)};
    \draw[line width=0.65pt, gray!50] plot [smooth, tension=0.5] coordinates{(-6,7.5) (-5.5, 7.5) (-5, 7.42) (-4.5, 7.25) (-4,7)};

    \draw [dashed, line width=1.2pt, lime!90] plot coordinates{(-6.425, -0.5) (-6.23,0.23)};
    \draw [line width=1.2pt, lime!90] plot coordinates{(-6.23,0.23) (-5, 5)};
    \draw [dashed, line width=1.2pt, lime!90] plot coordinates{(-5, 5) (-4.47, 7.2)};

    \draw[line width=1pt, black!50] plot coordinates{(-7, 1.95) (-7.5,2) (-8.5, -0.5) (-2.5,-1) (-1.5, 1.5) (-2, 1.55)};
    \node[rotate=0] at (-2.7, 1.6) {{\scriptsize $\mathbb{P}^{R-1}$}};
    \node[rotate=0] at (-8, 6) {{$\mathbb{P}^{R}$}};
    \node[rotate=0] at (-4, 6) {{\footnotesize $F$}};

    \draw [dashed, line width=0.5pt, blue] plot [smooth, tension=0.5] coordinates{(-6.8,0.82) (-6.9, 0.6) (-6.7, 0.3) (-6.5, 0.2) (-6.2, 0.25) };
    \draw [line width=0.5pt, blue] plot [smooth, tension=0.5] coordinates{(-6.2,0.25) (-5, 0.8) (-4.5, 1.25) (-4, 1.85) (-3.88,2.2)};
    \draw [dashed, line width=0.5pt, blue] plot [smooth, tension=0.5] coordinates{ (-3.88,2.2) (-4.5, 2.4) (-5, 2.5) (-5.5, 2.6) (-5.75, 2.7) (-6, 3) (-5.5, 3.55) (-5, 3.7) (-4.4, 3.5)};
    \draw [line width=0.5pt, blue] plot [smooth, tension=0.5] coordinates{(-4.4, 3.5) (-4.4, 3.25) (-4.5, 3) (-5, 2.5) (-5.5, 2) (-6, 1.7) (-6.5, 1.28) (-6.65, 1.1) (-6.8,0.82)};
    \node[blue] at (-4.41, 2.75) {{\tiny $X$}};

    \draw [fill=violet, line width=0pt] (-5,5) circle (2.8pt);
    \node[right] at (-5, 5) {{\tiny $P$}};
    \draw [fill=orange!50, line width=0pt] (-6.22,0.232) circle (1.5pt);
    \node[below] at (-6.12,0.232) {{\tiny $x$}};

    \draw[line width=1pt, gray!50] plot coordinates{(3.25,-0.5) (5.75,4.5)};
    \draw[line width=1pt, gray!50] plot coordinates{(7,-1) (4,5)};
    \draw [dashed, line width=0.5pt, gray!50] plot coordinates{(2.75, -1.5) (3.25,-0.5)};
    \draw [dashed, line width=0.5pt, gray!50] plot coordinates{(7.25, -1.5) (7,-1)};
    \draw [line width=1.5pt, red] plot [smooth, tension=0.5] coordinates{(3.25,-0.5) (4,-1) (5,-1.2) (5.5,-1.25) (6,-1.2) (7,-1)};
    \draw [dashed, line width=1.5pt, red] plot [smooth, tension=0.5] coordinates{(3.25,-0.5) (4,-0.25) (4.5,-0.2) (5,-0.25) (5.5,-0.35) (6,-0.5) (6.5,-0.7) (7,-1)};
    \node[red] at (6.65, -1.4) {{\tiny $Y^{\prime}$}};
    \draw[line width=1pt, gray!50] plot [smooth, tension=0.5] coordinates{(4,5) (4.5, 4.5) (5, 4.35) (5.5, 4.4) (5.75, 4.5)};
    \draw[line width=0.65pt, gray!50] plot [smooth, tension=0.5] coordinates{(4,5) (4.5, 5.05) (5, 4.95) (5.5, 4.7) (5.75, 4.5)};

    \draw [dashed, line width=1.2pt, orange!50] plot coordinates{(3.09, -1.3) (3.4,-0.6)};
    \draw [line width=1.2pt, orange!50] plot coordinates{(3.4,-0.6) (5, 3)};
    \draw [dashed, line width=1.2pt, orange!50] plot coordinates{(5, 3) (5.68, 4.57)};

    \draw[line width=1pt, black!50] plot coordinates{(3.2, 0.47) (3,0.5) (2.2, -1.5) (7.2,-2.0) (8, 0) (7.5, 0.05)};
    \node[rotate=0] at (7.2, 0.3) {{\scriptsize $\mathbb{P}^{R-2}$}};
    \node[rotate=0] at (1.5, 3.5) {{\small $H \cong \mathbb{P}^{R-1}$}};
    \node[rotate=0] at (5.8, 3.7) {{\footnotesize $F^{\prime}$}};

    \draw [line width=0.5pt, blue] plot [smooth, tension=0.5] coordinates{(5.25,4) (5,3.6) (4.95,3.3) (5,3)};
    \draw [dashed, line width=0.5pt, blue] plot [smooth, tension=0.5] coordinates{(5,3) (5.1, 2.7) (5.3, 2.3) (5.5, 2)};
    \draw [line width=0.5pt, blue] plot [smooth, tension=0.5] coordinates{(5.5,2) (5.3, 1.7) (5, 1.5) (4.5, 1.2) (4,1)};
    \draw [dashed, line width=0.5pt, blue] plot [smooth, tension=0.5] coordinates{(4,1) (4.5, 0.97) (5, 0.84) (5.5, 0.66) (6, 0.36) (6.5, 0)};
    \draw [line width=0.5pt, blue] plot [smooth, tension=0.5] coordinates{(6.5,0) (6.25, -0.3) (6, -0.5) (5.5, -0.8) (5, -1) (4.5, -1.07) (4,-1)};
    \draw [dashed, line width=0.5pt, blue] plot [smooth, tension=0.5] coordinates{(4,-1) (3.5, -1.1) (3.25, -1.3) (3, -1.5)};
    \node[blue] at (5.42, 1.45) {{\tiny $X^{\prime}$}};

    \draw [fill=violet, line width=0pt] (5,3) circle (2.8pt);
    \node[right] at (5, 3) {{\tiny $P^{\prime}$}};

    \draw[->, dashed] plot [smooth, tension=0.5] coordinates{(-1.5,0.5) (2.5,0)};
    \node[rotate=0] at (0.25, 0.65) {{\footnotesize $\pi_x$}};

    \draw [line width=1.5pt, gray!50, domain=-160:130, samples=100] plot ({5 + 5*cos(\x)}, {1.1 + 5*sin(\x)});
\end{tikzpicture}

  \caption{When $x \in X \cap Y$, the inner projection $\pi_x$ with center $x$ transforms $F \subset \mathbb{P}^R$ into $F^{\prime} \subset H$, and it also maps $X$ isomorphically to $X^{\prime}$ and $Y$ isomorphically to $Y^{\prime}$.}
  \label{Figure_1}
\end{figure}

\begin{proof}
The image under $\pi_x$ of a point $z \in F$ different from $x$ is the point $z^{\prime}$ where the line $\overline{xz}$ through $x$ and $z$ meets $H$. Thus, the image of any line $l$ from the ruling of $F$ will be the line $l^{\prime}$ in which the plane spanned by $x$ and $l$ intersects $H$. For a smooth curve $Z \subset F$ passing through $x$, the map ${\pi_x}_{|_Z} : Z \dashrightarrow H$ extends to a morphism, known as the inner projection of $Z$, by defining the image of $x$ to be the point where the tangent line $t_{Z,x}$ to $Z$ at $x$ meets the hyperplane $H$. Since the hyperplane $H \cong \mathbb{P}^R$ is general, such a curve will be isomorphic to its image $Z^{\prime} = {\pi_x}_{|_Z} (Z)$, provided $Z$ doesn't have trisecant lines. According to the construction in \cite[Theorem A]{CIK21}, the curve $Y \subset H_Y \cong \mathbb{P}^{R-1}$ is general of genus $\gamma$ and degree $g - 2\gamma + 1$ in $\mathbb{P}^{R-1}$, $R - 1 = g - 3\gamma + 1$. Hence, it doesn't have trisecant lines. From \cite[Proposition 8.4.3]{FOV99}, it follows that the rational map ${\pi_x}_{|_Y} \, : \, Y \dashrightarrow H$ extends to an embedding $\widetilde{{\pi_x}_{|_Y}} : Y \hookrightarrow H \cap H_Y \cong \mathbb{P}^{R-2}$. This proves statement (2). The proof of statement (1) is similar. If $X$ had a trisecant line, say $l$, then it couldn't be a line from the ruling of $F$ because $X$ meets every such line in exactly two points. Then the image of $l$ under the projection with center $P$ to $H_Y$ would be a trisecant to $Y$, which we saw to be impossible. Applying again \cite[Proposition 8.4.3]{FOV99}, we deduce part (1).

It remains to show statement (3). For any such $Z \subset F$ passing through $x$, the tangent plane $T_{F,x}$ to $F$ at $x$ contains $t_{Z,x}$ and also the line $l_x$ from the ruling of $F$ that passes through $x$. The plane $T_{F,x}$ meets $H$ in a line, say $l^{\prime}_x$, which we postulate to be the transform $\widetilde{\pi_x} (x)$ of $x$ under $\widetilde{\pi_x}$. Since $H_Y$ cutting $Y$ on $F$ is a general hyperplane passing through $x$, the image $Y^{\prime} = \widetilde{{\pi_x}_{|_Y}} (Y)$ must be contained in the linear subspace $H_Y \cap H$ of codimension 2 in $\mathbb{P}^R$. Consequently, the cone $F$ will be transformed into a cone $F^{\prime} \subset H \cong \mathbb{P}^{R-1}$ over $Y^{\prime}$. It is clear from the construction that the properties described in the proposition are satisfied.
\end{proof}

\begin{remark}
The line $l_x$ meets $X$ in two points. One of them is $x$, and let us say that $x^{\prime}$ is the second one. The transform $\widetilde{\pi_x}$ contracts $l_x \setminus \{x\}$ to the vertex $P^{\prime}$ of $F^{\prime}$. In particular, $\widetilde{\pi_x}(x^{\prime}) = P^{\prime}$. The "image" of $x$ is the line $l^{\prime}_x$ which meets $X^{\prime}$ in the points $P^{\prime}$ and $\widetilde{{\pi_x}_{|_X}}(x)$. It can be shown that $l^{\prime}_x$ is tangent to $X^{\prime}$ at $P^{\prime}$ as their intersection multiplicity is two (see \cite[Proposition 3]{CIK23}).
\end{remark}

The projection of $X^{\prime}$ from $P^{\prime}$ to the hyperplane in $H \cong \mathbb{P}^{R-1}$ that cuts $Y^{\prime}$ on $F^{\prime}$ gives a double covering morphism $\varphi^{\prime} : X^{\prime} \to Y^{\prime}$. Due to $X^{\prime}$ being an inner projection of $X$ with center $x$, we have that the embedding $X^{\prime} \hookrightarrow H \cong \mathbb{P}^{R-1}$ is induced by the line bundle
\[
\mathcal{O}_X (1) (-x) \cong \mathcal{O}_X (R_{\varphi} - x) \cong \omega_X \otimes (\varphi^{\ast} \omega_Y)^{\vee} (-x)
\]
as in \cite[Theorem 4.4]{CIK17}; see also Lemma \ref{SSec32:LemmaRSFam}. This geometric picture hints that we can reconstruct the family from \cite[Theorem 4.4]{CIK17} as a family of curves on elementary transformations of the ruled surfaces used in \cite[Theorem A]{CIK21}; see {\rm Figure \ref{Figure_2}} and {\rm Figure \ref{Figure_3}} in Section \ref{Section_4}.

Before proceeding with the rigorous construction of the component $\mathcal{D}^{\prime}_{d-1, g, R-1}$ and the proof of {\rm Theorem B}, we formulate and prove the following auxiliary result which will be used in the calculation of $h^0 (X^{\prime}, N_{X^{\prime} / \mathbb{P}^{R-1}})$.

\begin{proposition}\label{Sec2InnerProjProp}
Let $C$ be a smooth integral curve of genus $g$ properly contained in $\mathbb{P}^r$, $r \geq 3$. Let $H$  be a hyperplane in $\mathbb{P}^r$, and $x$ be a point on $C$. Consider the linear projection map $\pi_x : \mathbb{P} \dashrightarrow H$ with center $x$. Denote by $\phi \, : \, C \to H$ the morphism extending the rational map ${\pi_x}_{|_C} : C \dashrightarrow H$. Suppose that the image $\phi(C)$ is a smooth integral curve $D$ of genus $\gamma$ properly contained in $H$. Let  $R_{\phi}$ the ramification divisor of the morphism $\phi : C \to D$. Then
\begin{equation}\label{Sec2InnerProjSESNormBund}
    0 \to \mathcal{O}_C (R_{\phi}) \otimes \mathcal{O}_C (1) \otimes \mathcal{O}_C (2x)
      \to N_{C / \mathbb{P}^r}
      \to \phi^{\ast} N_{D / \mathbb{P}^{r-1}} \otimes \mathcal{O}_C (x)
      \to 0 \, ,
\end{equation}
where $N_{C / \mathbb{P}^r}$ is the normal bundle of $C$ in $\mathbb{P}^r$ and $N_{D / \mathbb{P}^{r-1}}$ is the normal bundle of $D$ in $H$.
\end{proposition}
\begin{proof}
The morphism $\phi$ is defined via the blow-up $\sigma \, : \, \widetilde{\mathbb{P}^r}(x) \to \mathbb{P}^r$ of $\mathbb{P}^r$ at $x$. Let $E$ be the exceptional divisor, $E = \sigma^{-1}(x) \cong \mathbb{P}^{r-1}$, and let $h$ be the hyperplane class of $\mathbb{P}^r$. Let $\sigma^{\ast}(h)$ denote its pull-back on $\widetilde{\mathbb{P}^r}(x)$. Denote by $\tilde{C}$ the strict transform of $C$ under the blow-up. The complete linear system $|\sigma^{\ast}(h) - E|$ defines the morphism $\tilde{\pi}_x : \widetilde{\mathbb{P}^r}(x) \to \mathbb{P}^{r-1}$ extending the rational map $\pi_x$ (see \cite[Chapter 2, Example 7.17.3]{Hart77}). The isomorphism $C \cong \tilde{C}$ followed by ${\tilde{\pi}_x}|_{\tilde{C}}$ defines the morphism $\phi : C \to D$ extending ${\pi_x}_{|_C} : C \dashrightarrow H$. Since $\tilde{C}$ meets $E$ at a single point, it follows that $\phi^{\ast} \mathcal{O}_D (1) \cong \mathcal{O}_C (1) \otimes \mathcal{O}_C (-x)$ on $C$.

Consider the Euler sequences
\[
     0 \to \mathcal{O}_C
       \to \oplus^{r+1}_{1} \mathcal{O}_C (1)
       \to T_{\mathbb{P}^r |_C} \to 0 \,
\]
and
\[
     0 \to \mathcal{O}_D
       \to \oplus^{r}_{1} \mathcal{O}_D (1)
       \to T_{\mathbb{P}^{r-1} |_D}
       \to 0 \, .
\]

\noindent Pulling back the second exact sequence to $C$ via $\phi^{\ast}$ we obtain

{\small
\begin{equation*}
\begin{array}{ccccccccccccccccccccccc}
  & & & & 0 & & 0 & & \\[1ex]
  & &  & & \downarrow & & \downarrow & & \\ [1ex]
  & & 0 & & \mathcal{O}_C(1) &  & \ker(\alpha) & &  \\[1ex]
  & & \downarrow & & \downarrow & & \downarrow & \\[1ex]
  0 & \rightarrow & \mathcal{O}_{C} & \rightarrow & \oplus^{r+1}_{1} \mathcal{O}_C (1) & \rightarrow & T_{\mathbb{P}^r |_C} &\
\rightarrow & 0 \\[1ex]
  & & \downarrow & & \downarrow & & \downarrow \alpha & \\[1ex]
  0 & \rightarrow 	& \mathcal{O}_{C}(x) & \rightarrow
		& {\phi}^{\ast} \left( \oplus^{r}_{1} \mathcal{O}_{D}(1) \right) \otimes \mathcal{O}_{C} (x)		& \rightarrow
		& {\phi}^{\ast} \left( T_{\mathbb{P}^{r-1} |_{D}}\right)   \otimes \mathcal{O}_{C} (x)
	& \rightarrow & 0 \\[1ex]
  & & \downarrow & & \downarrow & & \downarrow & \\[1ex]
  & & \mathcal{O}_{x} & & 0 & &  0 & & \\[1ex]
  & & \downarrow & & & & & \\[1ex]
  & & 0 & & & & & \\[1ex]
\end{array}
\end{equation*}
}
\noindent By chasing the above diagram, we obtain from the Snake Lemma that $\ker(\alpha) \cong \mathcal{O}_{C} (1) \otimes \mathcal{O}_{C} (x)$,
whence
\[
 0 	\to \mathcal{O}_{C} (1) \otimes \mathcal{O}_{C} (x)
	\to T_{\mathbb{P}^r |_C}
	\to  {\phi}^{\ast} \left( T_{\mathbb{P}^{r-1} |_{D}}\right)   \otimes \mathcal{O}_{C} (x)
	\to 0 \, .
\]

\noindent Further, using the canonical sequence for $N_{C / \mathbb{P}^r}$
\[
  0 \to T_C \to {T_{\mathbb{P}^r}}_{|_C} \to N_{C / \mathbb{P}^r} \to 0
\]
and the one for $N_{D / \mathbb{P}^{r-1}}$ tensored by $\mathcal{O}_{C} (x)$
\[
  0 	\to {\phi}^{\ast}(T_{D}) \otimes \mathcal{O}_{C} (x)
	\to {\phi}^{\ast}({T_{\mathbb{P}^{r-1}}}_{|_{D}}) \otimes \mathcal{O}_{C} (x)
	\to {\phi}^{\ast}(N_{D / \mathbb{P}^{r-1}}) \otimes \mathcal{O}_{C} (x)
	\to 0 \, ,
\]
we get the commutative diagram
\begin{equation*}
\begin{array}{ccccccccccccccccccccccc}
  & & & & 0 & & 0 & & \\[1ex]
  & & & & \downarrow & & \downarrow & & \\[1ex]
  & & 0 &   & \mathcal{O}_C(1) \otimes \mathcal{O}_C (x) &  & \ker(\beta) & &  \\[1ex]
  & & \downarrow & & \downarrow & & \downarrow & \\[1ex]
  0 & \rightarrow & T_{C} & \rightarrow
      & T_{\mathbb{P}^r |_C} & \rightarrow & N_{C / \mathbb{P}^r} & \rightarrow & 0 & \\[1ex]
  & & \downarrow & & \downarrow & & \downarrow \beta & \\[1ex]

  0  &	\to  &  {\phi}^{\ast}(T_{D}) \otimes \mathcal{O}_{C} (x)
 &	\to   & {\phi}^{\ast}({T_{\mathbb{P}^{r-1}}}_{|_{D}}) \otimes \mathcal{O}_{C} (x)
 &	\to  & {\phi}^{\ast}(N_{D / \mathbb{P}^{r-1}}) \otimes \mathcal{O}_{C} (x)  &	\to & 0
\\[1ex]
    & & \downarrow & & \downarrow &  & \downarrow & \\[1ex]
  & & \mathcal{O}_{R_{ \phi }+x} & & 0 & & 0 & & \\[1ex]
  & & \downarrow & & &  & & \\[1ex]
  & & 0 & & & & & \\[1ex]
\end{array}
\end{equation*}

\noindent Similarly, as before, $\ker \beta \cong \mathcal{O}_C(1) \otimes \mathcal{O}_{C} (R_{ \phi }+2x)$, from which we obtain the short exact sequence
\[
  0   \to \mathcal{O}_{C} (1) \otimes \mathcal{O}_{C} (R_{ \phi }+2x)
      \to N_{C / \mathbb{P}^r}
      \to {\phi}^{\ast}N_{D / \mathbb{P}^{r-1}} \otimes \mathcal{O}_{C} (x)
      \to 0 \, .
\]
\end{proof}

\begin{remark}
When $\phi$ is an isomorphism, the claim can be derived from \cite[Lemma 4]{Ein87}; see also \cite[Lemma 6]{Wal94} and \cite[Proposition 6]{FS19}. We will use {\rm Proposition \ref{Sec2InnerProjProp}} in the proof of {\rm Theorem B} in the case of $\phi$ being a double covering.
\end{remark}

\section{Ruled surfaces and proof of Theorem B}\label{Section_3}

Proposition \ref{Sec2:PropInnerProj} provided the geometric motivation for the appearance of the curves parametrized by the component $\mathcal{D}^{\prime}_{d-1, g, R-1}$ in Theorem B. However, the proof of the theorem requires a more algebraic construction of the corresponding family of curves in order to compute its dimension and deduce additional properties. We remark that the cones considered in Proposition \ref{Sec2:PropInnerProj} can be viewed as decomposable ruled surfaces, which we will use. For this reason, we will first recall some facts about decomposable ruled surfaces, after which we will proceed with the proof of Theorem B.

\subsection{A few facts about ruled surfaces}\label{Sec2sub1}

Let $F \subset \mathbb{P}^{r+1}$, $r \geq 3$, be a cone over a smooth non-degenerate integral curve $\Gamma \subset \mathbb{P}^r$ of genus $\gamma > 0$ and degree $e > 2\gamma$, such that $r = e - \gamma$ (this means that the embedding $\Gamma \subset \mathbb{P}^r$ is complete). Such a cone can be viewed as the decomposable ruled surface $S = \mathbb{P}(\mathcal{E})$ with $\mathcal{E} = \mathcal{O}_{\Gamma} \oplus \mathcal{O}_{\Gamma}(-1)$. Recall that for any line bundle $\mathcal{L}$ on $\Gamma$ we have $\mathbb{P}(\mathcal{E}) \cong \mathbb{P}(\mathcal{E} \otimes \mathcal{L})$. Note that $\mathcal{E}$ is normalized in the sense that for every line bundle $\mathcal{L}$ with $\deg \mathcal{L} < 0$, we have $H^0(\mathcal{E} \otimes \mathcal{L}) = 0$ while $H^0(\mathcal{E}) \neq 0$. Also, there is a one-to-one correspondence between sections of $f : S = \mathbb{P}(\mathcal{E}) \to \Gamma$ and surjections $\mathcal{E} \twoheadrightarrow L$, with $L$ a line bundle on $\Gamma$. In particular, if $N = \ker(\mathcal{E} \twoheadrightarrow L)$ and $B$ is the section corresponding to the exact sequence
\[
0 \to N \to \mathcal{E} \to L \to 0,
\]
then $B^2 = \deg L - \deg N$. For this reason, the exact sequence
\[
0 \to \mathcal{O}_{\Gamma} \to \mathcal{O}_{\Gamma} \oplus \mathcal{O}_{\Gamma}(-1) \to \mathcal{O}_{\Gamma}(-1) \to 0
\]
yields a section, say $\Gamma_0$, whose self-intersection number is $\Gamma^2_0 = \deg \mathcal{O}_{\Gamma}(-1) = -e$. It is the minimal possible self-intersection number a section of $f : S = \mathbb{P}(\mathcal{O}_{\Gamma} \oplus \mathcal{O}_{\Gamma}(-1)) \to \Gamma$ can have. As $\mathcal{E}$ is decomposable, the section $\Gamma_0$ of minimal self-intersection is unique. The Picard group of $S$ is $\Pic(S) \cong \mathbb{Z}[\mathcal{O}_S(\Gamma_0)] \oplus f^{\ast}(\Pic(\Gamma))$. For a divisor $D$ on $\Gamma$, denote by $D\mathfrak{f}$ the divisor $f^{\ast}(D)$ on $S$. For every effective divisor $E \in |\mathcal{O}_{\Gamma}(1)|$, the exact sequence
\[
0 \to \mathcal{O}_{\Gamma}(-1) \to \mathcal{O}_{\Gamma} \oplus \mathcal{O}_{\Gamma}(-1) \to \mathcal{O}_{\Gamma} \to 0
\]
determines a section $\Gamma_1 \sim \Gamma_0 + E\mathfrak{f}$. As shown in \cite[Prop. 1]{CIK21}, the linear series $|\mathcal{O}_S(\Gamma_0 + E\mathfrak{f})|$ determines a morphism $\Psi : S \to \mathbb{P}^{r+1}$, $r = e - \gamma$, that maps $S$ into the cone $F$ and contracts $\Gamma_0$ to the vertex of $F$.

\begin{prop}\label{SSec31:LSonRS}
Let $\Gamma$, $S$ and $E$ be as above. Let $q \in \Gamma$ be a point and $m > 0$ be an integer. Then
\begin{itemize}
 \item[{\rm (a)}] the linear series $|\mathcal{O}_S (m\Gamma_0 + mE \mathfrak{f})|$ is base-point-free, $\dim |\mathcal{O}_S (m\Gamma_0 + mE \mathfrak{f})| = \binom{m+1}{2}e - m\gamma + m$ and the image of an effective divisor $C_m \in |\mathcal{O}_S (m\Gamma_0 + mE \mathfrak{f})|$ is a curve cut by a degree $m$ hypersurface in $\mathbb{P}^{r+1}$;

 \item[{\rm (b)}] the linear series $|\mathcal{O}_S (m\Gamma_0 + (mE+q) \mathfrak{f})|$ has a unique base-point $\Gamma_0 \cap q\mathfrak{f}$, $\dim |\mathcal{O}_S (m\Gamma_0 + (mE +q)\mathfrak{f})| = \binom{m+1}{2}e - m\gamma + 2m$ and the image of a general effective divisor $C^{\prime}_m \in |\mathcal{O}_S (m\Gamma_0 + (mE + q) \mathfrak{f})|$ is a smooth curve on $F$ passing through the vertex.
\end{itemize}
\end{prop}
\begin{proof}
Both statements regarding the dimensions of the linear series can be easily deduced, given that $S$ is decomposable. The dimensions are obtained by applying the formula that for a divisor $B$ on $\Gamma$,
\begin{equation}\label{SSec31:h0OS}
h^0(S, \mathcal{O}_S(m \Gamma_0 + B \mathfrak{f})) = \sum_{k=0}^m h^0(\Gamma, \mathcal{O}_{\Gamma}(B - kE)),
\end{equation}
as shown in \cite[Lemma 35]{FP05}. Thus, the claim in {\rm (a)} is a trivial generalization of the claim in \cite[Prop. 1]{CIK21}. The claim in {\rm (b)} follows from \cite[Prop. 36]{FP05}, which states that $\Gamma_0 \cap q\mathfrak{f}$ is the unique base point of $|m\Gamma_0 + (mE + q) \mathfrak{f}|$ on $S$ and that a general element of that linear series is a smooth integral curve. Since $\Psi$ contracts $\Gamma_0$ to the vertex of $F$, the image of $C^{\prime}_m$ must pass through the vertex.
\end{proof}

\subsection{Proof of Theorem B}\label{Sec2sub2}

We recall the basic numerical assumptions in the theorem:
\[
 \gamma \geq 7, \quad g \geq 6\gamma + 5, \quad d = 2g - 4\gamma + 2 \quad \mbox{ and } \quad R = g - 3\gamma + 2.
\]

The proof proceeds in three steps:

\subsection*{Step I: Construction of the family \texorpdfstring{$\mathcal{F}^{\prime}$}{F'} and \texorpdfstring{$\mathcal{D}^{\prime}_{d-1, g, R-1}$}{D' d-1, g, R-1}}

In this step, we describe the construction of the family \(\mathcal{F}^{\prime}\) and define \(\mathcal{D}^{\prime}_{d-1, g, R-1}\). We show that
\[
\dim \mathcal{D}^{\prime}_{d-1, g, R-1} = 2g - \gamma - 1 + R^2.
\]

\subsection*{Step II: Dimension of the space of first-order deformations}

In this step, we compute the dimension of the space of first-order deformations of a general curve \(X^{\prime}\) from the family \(\mathcal{F}^{\prime}\). We demonstrate that
\[
\dim T_{[X^{\prime}]} \mathcal{D}^{\prime}_{d-1, g, R-1} = h^0(X^{\prime}, N_{X^{\prime} / \mathbb{P}^{R-1}})
= 2g - \gamma + R^2 = \dim \mathcal{D}^{\prime}_{d-1, g, R-1} + 1.
\]

\subsection*{Step III: Non-reduced component \texorpdfstring{$\mathcal{D}^{\prime}_{d-1, g, R-1}$}{D' d-1, g, R-1}}

In this step, we show that no component of \(\mathcal{I}_{d-1, g, R-1}\) contains \(\mathcal{D}^{\prime}_{d-1, g, R-1}\) properly. This implies that \(\mathcal{D}^{\prime}_{d-1, g, R-1}\) is a non-reduced component of \(\mathcal{I}_{d-1, g, R-1}\).

\medskip

The claims formulated in {\rm Theorem B} are obtained in the course of the completion of those three steps. The characterization  of the curves parametrized by $\mathcal{D}^{\prime}_{d-1, g, R-1}$ as in item {\it (iii)} of {\rm Theorem B} is given in {\rm Step I}.

\medskip

\subsection*{Step I: Construction of the family \texorpdfstring{$\mathcal{F}^{\prime}$}{F'} and \texorpdfstring{$\mathcal{D}^{\prime}_{d-1, g, R-1}$}{D' d-1, g, R-1}}

Consider a general curve $\Gamma \in \mathcal{M}_{\gamma}$ of genus $\gamma$, a general divisor  $E^{\prime}$ of degree $g - 2\gamma$ on $\Gamma$ and a point $q \in \Gamma$. Let $S^{\prime}$ be the ruled surface $S^{\prime} := \mathbb{P} (\mathcal{O}_{\Gamma} \oplus \mathcal{O}_{\Gamma}(-E^{\prime}))$ with natural projection morphism $f^{\prime} : S \to \Gamma$. Let $\Gamma^{\prime}_0$ be the section corresponding to the surjective morphism $\mathcal{O}_{\Gamma} \oplus \mathcal{O}_{\Gamma}(-E^{\prime}) \twoheadrightarrow  \mathcal{O}_{\Gamma}(-E^{\prime})$.  Consider the morphism $\Psi^{\prime} : S^{\prime} \to \mathbb{P}^{R-1}$. Now, $\Pic (S^{\prime}) \cong \mathbb{Z} [\mathcal{O}_{S^{\prime}} (\Gamma^{\prime}_0)] \oplus {f^{\prime}}^{\ast} (\Pic (\Gamma))$ and we denote by $D\mathfrak{f^{\prime}}$ the divisor ${f^{\prime}}^{\ast} (D)$ on $S^{\prime}$ for a divisor $D$ on $\Gamma$. We have for the intersection numbers $\Gamma^{\prime}_0 \cdot \Gamma^{\prime}_0 = \deg \mathcal{O}_{\Gamma}(-E^{\prime}) = - (g - 2\gamma)$, $\Gamma^{\prime}_0 \cdot \mathfrak{f^{\prime}} = 1$ and $\mathfrak{f^{\prime}} \cdot \mathfrak{f^{\prime}} = 0$.

Define $\mathcal{F}^{\prime}$ as the family of curves that are images of the divisors from the linear series $|2{\Gamma}^{\prime}_0 + (2E^{\prime} + q){\mathfrak{f}^{\prime}}|$ on $S^{\prime}$ under the morphism $\Psi^{\prime}$, by varying $\Gamma$ in $\mathcal{M}_{\gamma}$, running $E$ through the set of general effective divisors of degree $g-2\gamma$ on $\Gamma$ and $q \in \Gamma$. Denote by $X^{\prime}$ a general element of the family $\mathcal{F}^{\prime}$. In particular, $X^{\prime} = \Psi^{\prime} (C^{\prime})$ for a general $C^{\prime} \in |2\Gamma^{\prime}_0 + (2E^{\prime} + q)\mathfrak{f}^{\prime}|$. Let $D^{\prime} \in |\Gamma^{\prime}_0 + E^{\prime}\mathfrak{f}^{\prime}|$ be general and $Y^{\prime} = \Psi^{\prime} (D^{\prime})$.

\begin{lemma}\label{SSec32:LemmaRSFam}
Denote $F^{\prime} = \Psi^{\prime} (S^{\prime})$. Then
\begin{enumerate}[label=(\alph*), leftmargin=*, font=\rmfamily]
 \item $Y^{\prime}$ is a smooth curve of degree $g-2\gamma$ and genus $\gamma$, properly contained in a hyperplane $H^{\prime} \cong \mathbb{P}^{R-2} \subset \mathbb{P}^{R-1}$;

 \item $F^{\prime}$ is a cone over $Y^{\prime}$ with vertex $P^{\prime} = \Psi^{\prime} (\Gamma^{\prime}_0)$;

 \item $X^{\prime}$ is a linearly normal smooth curve on $F^{\prime}$ of degree $d = 2g - 4\gamma + 1$ and genus $g$ that passes through $P^{\prime}$;

 \item the linear projection $\pi_{P^{\prime}} : X^{\prime} \dashrightarrow H^{\prime}$ extends to a double-covering morphism $\varphi^{\prime} : X^{\prime} \to Y^{\prime}$ for which $\mathcal{O}_{X^{\prime}} (1) \cong \mathcal{O}_{X^{\prime}} (R_{\varphi^{\prime}} - \zeta)$, where $R_{\varphi^{\prime}}$ is  the ramification divisor of $\varphi^{\prime}$ and $\zeta$ is the other point, apart from $P^{\prime}$, in which the tangent line
 $t_{P^{\prime}}$ to $X^{\prime}$ at $P^{\prime}$ meets $X^{\prime}$.
\end{enumerate}
\end{lemma}
\begin{proof}[Proof of the lemma]
Most of the claims are clear from the preceding discussion. By Bertini's theorem $C^{\prime}$ and $D^{\prime}$ are smooth curves and since $\Psi^{\prime}$ is isomorphism away from $\Gamma^{\prime}_0$, it follows that $X^{\prime}$ and $Y^{\prime}$ are also smooth curves of degree $\deg X^{\prime} = (2\Gamma^{\prime}_0 + (2E^{\prime} + q)\mathfrak{f}^{\prime}) \cdot (\Gamma^{\prime}_0 + E^{\prime} \mathfrak{f}^{\prime}) = 2g-4\gamma+1$ and $\deg Y^{\prime} = (\Gamma^{\prime}_0 + E^{\prime} \mathfrak{f}^{\prime}) \cdot (\Gamma^{\prime}_0 + E^{\prime} \mathfrak{f}^{\prime}) = g-2\gamma$, correspondingly. For canonical divisor $K_{S^{\prime}}$ of $S^{\prime}$ we have $K_{S^{\prime}} \sim -2\Gamma^{\prime}_0 + (K_{\Gamma} - E^{\prime})\mathfrak{f}^{\prime}$, where $K_{\Gamma}$ denotes the canonical divisor of $\Gamma$. Applying the adjunction formula we get $(K_{S^{\prime}}) \cdot C^{\prime} = ((K_{\Gamma} + E^{\prime}+q)\mathfrak{f}^{\prime}) \cdot (2\Gamma^{\prime}_0 + (2E^{\prime}+q)\mathfrak{f}^{\prime}) = 2(g-1)$, which implies that the genus of $C^{\prime}$,  is $g$, and hence of $X^{\prime}$ as well. The argument that the genus of $D^{\prime}$, and hence of $Y^{\prime}$, is $\gamma$ is similar. The fact that $Y^{\prime}$ is properly contained in hyperplane of $\mathbb{P}^{R}$ follows by Riemann-Roch theorem since $h^0 (\mathcal{O}_{Y^{\prime}} (1)) = h^0 (\mathcal{O}_{D^{\prime}} (\Gamma^{\prime}_0 + E^{\prime}\mathfrak{f}^{\prime})) = g-2\gamma - \gamma + 1 = R-1$, due to the assumption that $g > 6\gamma + 5$. This completes the proof of {(a)}. Further, formula {\rm (\ref{SSec31:h0OS})} gives $h^0 (\mathcal{O}_{S^{\prime}} (\Gamma^{\prime}_0 + E^{\prime}\mathfrak{f}^{\prime})) = R$, so $F^{\prime}$ is linearly normal in $\mathbb{P}^{R-1}$. Since $D^{\prime}$ is a divisor in the linear system determining the morphism $\Psi^{\prime}$ and $Y^{\prime} = \Psi^{\prime} (D^{\prime})$, we can regard $F^{\prime}$ as a cone over $Y^{\prime}$. This gives claim {(b)}. To deduce the linear normality of $X^{\prime} \subset \mathbb{P}^{R-1}$ is sufficient to show that $h^0 (\mathcal{O}_{X^{\prime}} (1)) = h^0 (\mathcal{O}_{C^{\prime}} (\Gamma^{\prime}_0 + E^{\prime}\mathfrak{f}^{\prime})) = g - 3\gamma + 2$. For this consider the short exact sequence
\[
 0 \to \mathcal{O}_{S^{\prime}} (-\Gamma^{\prime}_0 - (E^{\prime}+q)\mathfrak{f}^{\prime}) \to \mathcal{O}_{S^{\prime}} (\Gamma^{\prime}_0 + E^{\prime}\mathfrak{f}^{\prime}) \to \mathcal{O}_{C^{\prime}} (\Gamma^{\prime}_0 + E^{\prime}\mathfrak{f}^{\prime}) \to 0 \, .
\]
Using Serre duality and Kawamata-Wiehweg vanishing theorem, it is not difficult to deduce that $h^1 (\mathcal{O}_{S^{\prime}} (-\Gamma^{\prime}_0 - (E^{\prime}+q)\mathfrak{f}^{\prime}))) = 0$, hence $h^0 (\mathcal{O}_{C^{\prime}} (\Gamma^{\prime}_0 + E^{\prime}\mathfrak{f}^{\prime})) = h^0 (\mathcal{O}_{S^{\prime}} (\Gamma^{\prime}_0 + E^{\prime}\mathfrak{f}^{\prime})) = R$. Since the linear system $|2\Gamma^{\prime}_0 + (2E^{\prime} + q)\mathfrak{f}^{\prime}))|$ has a base point $\Gamma^{\prime}_0 \cap q\mathfrak{f}^{\prime}$, it follows that $X^{\prime}$ will pass through the vertex $P^{\prime}$ of $F^{\prime}$. This completes the proof of {\rm (c)}.

It remains to prove {\rm (d)}. The linear projection $\pi_{P^{\prime}} : X^{\prime} \dashrightarrow H^{\prime}$ extends to a morphism via the blow-up of $F^{\prime}$ at $P^{\prime}$, but this is exactly the double covering morphism $\phi : C^{\prime} \to D^{\prime}$ on $S^{\prime}$ induced by the ruling. Thus, $\varphi : X^{\prime} \to Y^{\prime}$ is a double covering morphism. Since $C^{\prime} \sim 2\Gamma^{\prime}_0 + (2E^{\prime} + q)\mathfrak{f}^{\prime}$ is two-secant on $S^{\prime}$, the fiber $q\mathfrak{f}^{\prime}$ meets $C^{\prime}$ in two points: one of them is $x^{\prime} = C^{\prime} \cap \Gamma^{\prime}_0$ and let $z$ be the other. The morphism $\Psi^{\prime}$ then maps the fiber $q\mathfrak{f}^{\prime}$ into a line from the ruling of $F^{\prime}$ that is tangent to $X^{\prime}$, or otherwise $\Gamma^{\prime}_0$ and $C^{\prime}$, which is not the case. The intersection multiplicity of $t_{P^{\prime}}$ and $X^{\prime}$ is exactly two since $t_{P^{\prime}}$ meets $X^{\prime}$ also in the point $\zeta = \Psi^{\prime} (z)$, which is different from $P^{\prime}$. It remains to show only that $\mathcal{O}_{X^{\prime}} (1) \cong \mathcal{O}_{X^{\prime}} (R_{\varphi^{\prime}} - \zeta)$, which, by the linear normality of $X^{\prime}$, is equivalent to showing that the linear series $|\mathcal{O}_{S^{\prime}} (\Gamma^{\prime}_0 + E^{\prime}\mathfrak{f}^{\prime})|$ on $S^{\prime}$ induces the linear series $|\mathcal{O}_{C^{\prime}}(R_{\phi} - z)|$ on $C^{\prime}$, where $R_{\phi}$ is the ramification divisor of $\phi$. We have $R_{\phi} \sim K_{C^{\prime}} - \phi^{\ast} K_{D^{\prime}} \sim (K_{S^{\prime}} + C^{\prime})_{|_{C^{\prime}}} - (K_{S^{\prime}} + D^{\prime})_{|_{C^{\prime}}} \sim ((E^{\prime} + q)\mathfrak{f}^{\prime})_{|_{C^{\prime}}} \sim E^{\prime} \mathfrak{f}^{\prime}_{|_{C^{\prime}}} + x^{\prime} +z$. On the other hand, $(\Gamma^{\prime}_0 + E^{\prime} \mathfrak{f}^{\prime})_{|_{C^{\prime}}} \sim x^{\prime} + E^{\prime} \mathfrak{f}^{\prime}_{|_{C^{\prime}}}$, therefore $R_{\phi} - z \sim (\Gamma^{\prime}_0 + E^{\prime} \mathfrak{f}^{\prime})_{|_{C^{\prime}}}$. This completes the proof.
\end{proof}

The lemma gives that a general $X^{\prime}$ from the family $\mathcal{F}^{\prime}$ satisfies the properties {\rm (iii).1}, {\rm (iii).2} and {\rm (iii).3} of the theorem. For the dimension of $\mathcal{F}^{\prime}$ we have:

\noindent $\dim \mathcal{F} = $
\begin{itemize}[font=\sffamily, leftmargin=1.3cm, style=nextline]
     \item[$ + $] $3\gamma - 3$ \ : \ number of parameters of curves $\Gamma \in \mathcal{M}_{ \gamma }$

     \item[$ + $] $\gamma$ \ : \ number of parameters of line bundles $\mathcal{O}_{\Gamma} (E^{\prime}) \in \Pic (\Gamma)$ of degree $g-2\gamma \geq 4\gamma + 5$
     necessary to fix the geometrically ruled surface $\mathbb{P} (\mathcal{O}_{\Gamma} \oplus \mathcal{O}_{\Gamma} (-E^{\prime}))$

     \item[$ + $] $R^2 - 1 = \dim (\Aut (\mathbb{P}^{R-1}))$

     \item[$ + $] $1$ \ : \ number of parameters necessary to fix $q \in \Gamma$

     \item[$ - $] $((g - 2\gamma) - \gamma + 2) = \dim G_{F}$, where $G_{F}$ is the subgroup of $\Aut (\mathbb{P}^{R-1})$ fixing the scroll $F^{\prime}$, see \cite[Lemma 6.4, p. 148]{CCFM2009}

     \item[$ + $] $3g - 8\gamma + 4 = \dim |2{\Gamma}^{\prime}_0 + (2E^{\prime} + q){\mathfrak{f}^{\prime}}|$ \ : \ number of parameters to choose a curve in the linear equivalence class of $2{\Gamma}^{\prime}_0 + (2E^{\prime} + q){\mathfrak{f}^{\prime}}$ on $S^{\prime}$.
\end{itemize}
Accounting the numbers we get $\dim \mathcal{F}^{\prime} = 2g - \gamma - 1 + R^2$. Finally, define $\mathcal{D}^{\prime}_{d-1, g, R-1} \subset \mathcal{I}_{d-1, g, R-1}$ to be the closure, in $\mathcal{I}_{d-1, g, R-1}$, of the set parametrizing the family $\mathcal{F}^{\prime}$. In particular,
\[
 \dim \mathcal{D}^{\prime}_{d-1, g, R-1} = 2g - \gamma - 1 + R^2 .
\]
This completes {\rm Step I}.

\medskip

\subsection*{Step II: Dimension of the space of first-order deformations}

As it is well known, the space of first-order deformations of $X^{\prime}$ inside $\mathbb{P}^{R-1}$ is given by $H^0 (X^{\prime}, N_{X^{\prime} / \mathbb{P}^{R-1}})$, whose dimension we will compute in this step. By the construction of the family $\mathcal{F}^{\prime}$ in Step I, we have the double covering morphism $\varphi^{\prime} : X^{\prime} \to Y^{\prime}$. Applying {\rm Proposition \ref{Sec2InnerProjProp}}, we obtain the short exact sequence
{\small
\begin{equation}\label{Sec3SESInnerProjXPrime}
  0 \to \mathcal{O}_{X^{\prime}} (1) \otimes \mathcal{O}_{X^{\prime}} (R_{{\varphi}^{\prime}} + 2P^{\prime})
      \to N_{X^{\prime} / \mathbb{P}^{R-1}}
      \to {{\varphi}^{\prime}}^{\ast} N_{Y^{\prime} / \mathbb{P}^{R-2}} \otimes \mathcal{O}_{X^{\prime}} (P^{\prime})
      \to 0 \, ,
\end{equation}
}
which we will use to compute $h^0 (X^{\prime}, N_{X^{\prime} / \mathbb{P}^{R-1}})$. Since $\deg \mathcal{O}_{X^{\prime}} (1) \otimes \mathcal{O}_{X^{\prime}} (R_{{\varphi}^{\prime}} + 2P^{\prime}) = (d-1) + (2g-2 - 2(2\gamma - 2)) + 2 = (2g - 4\gamma + 1) + 2g - 4\gamma + 4 = 4g - 8\gamma + 5 > 2g - 2$, the line bundle $\mathcal{O}_{X^{\prime}} (1) \otimes \mathcal{O}_{X^{\prime}} (R_{{\varphi}^{\prime}} + 2P^{\prime})$ is non-special. Hence, the sequence {\rm (\ref{Sec3SESInnerProjXPrime})} is exact on global section and we have
\[
 h^0 (N_{X^{\prime} / \mathbb{P}^{R-1}}) = h^0 (\mathcal{O}_{X^{\prime}} (1) \otimes \mathcal{O}_{X^{\prime}} (R_{{\varphi}^{\prime}} + 2P^{\prime})) + h^0 ({{\varphi}^{\prime}}^{\ast} N_{Y^{\prime} / \mathbb{P}^{R-2}} \otimes \mathcal{O}_{X^{\prime}} (P^{\prime})) \, .
\]
By Riemann-Roch theorem we have
\begin{equation}\label{SSec32Eqh0NXP}
 h^0 (X^{\prime}, \mathcal{O}_{X^{\prime}} (1) \otimes \mathcal{O}_{X^{\prime}} (R_{{\varphi}^{\prime}} + 2P^{\prime}))
 = 3g - 8\gamma + 6 \, ,
\end{equation}
while for the computation of $h^0 (X^{\prime}, {{\varphi}^{\prime}}^{\ast} N_{Y^{\prime} / \mathbb{P}^{R-2}} \otimes \mathcal{O}_{X^{\prime}} (P^{\prime}))$ we use that
\begin{equation}\label{SSec32Eqh0NYP}
\begin{aligned}
 h^0 (X^{\prime}, {{\varphi}^{\prime}}^{\ast} N_{Y^{\prime} / \mathbb{P}^{R-2}} \otimes \mathcal{O}_{X^{\prime}} (P^{\prime}))
 & = h^0 (Y^{\prime}, {\varphi}^{\prime}_{\ast} ( \, {{\varphi}^{\prime}}^{\ast} N_{Y^{\prime} / \mathbb{P}^{R-2}} \otimes \mathcal{O}_{X^{\prime}} (P^{\prime}) \, ) ) \\
 & = h^0 (Y^{\prime}, N_{Y^{\prime} / \mathbb{P}^{R-2}} \otimes {\varphi}^{\prime}_{\ast} \mathcal{O}_{X^{\prime}} (P^{\prime}) ) \, .
\end{aligned}
\end{equation}

\begin{lemma}\label{SSec32St2LemPF}
${\varphi}^{\prime}_{\ast} \mathcal{O}_{X^{\prime}} (P^{\prime}) \cong \mathcal{O}_{Y^{\prime}} \oplus \mathcal{O}_{Y^{\prime}} (-1)$.
\end{lemma}
\begin{proof}[Proof of the lemma.]
Recall that by the construction of $\mathcal{F}^{\prime}$ in Step I, $X^{\prime}$ is the isomorphic image of a general $C^{\prime} \sim 2{\Gamma^{\prime}}_0 + (2E^{\prime}+q)\mathfrak{f}^{\prime}$ on $S^{\prime}$, and $Y^{\prime}$ is the isomorphic image of a general $D^{\prime} \sim {\Gamma^{\prime}}_0 + E^{\prime}\mathfrak{f}^{\prime}$ on $S^{\prime}$. The inner projection $\varphi^{\prime} : X^{\prime} \to Y^{\prime}$ is defined via the double-covering morphism $\phi : C^{\prime} \to D^{\prime}$ induced by the ruling of $S^{\prime}$. The statement is then equivalent to
\begin{equation}\label{SSec32LemEq}
 \phi_{\ast} \mathcal{O}_{C^{\prime}} (x^{\prime}) \cong \mathcal{O}_{D^{\prime}} \oplus \mathcal{O}_{D^{\prime}} (-{\Gamma^{\prime}}_0 - E^{\prime} \mathfrak{f}^{\prime}) = \mathcal{O}_{D^{\prime}} \oplus \mathcal{O}_{D^{\prime}} (-E^{\prime} \mathfrak{f}^{\prime}) \, .
\end{equation}
Since $D^{\prime} \cong \Gamma$, we conclude, using \cite[Proposition 2.2]{FP05MN}, that
\[
 \mathcal{O}_{D^{\prime}} \oplus \mathcal{O}_{D^{\prime}} (-E^{\prime} \mathfrak{f}^{\prime}) \cong \phi_{\ast} \mathcal{O}_{C^{\prime}} ({\Gamma^{\prime}}_0) \, .
\]
Finally, using that ${{\Gamma^{\prime}}_0}_{|_{C^{\prime}}} \sim x^{\prime}$, we get $\phi_{\ast} \mathcal{O}_{C^{\prime}} ({\Gamma^{\prime}}_0) = \phi_{\ast} \mathcal{O}_{C^{\prime}} (x^{\prime})$. This completes the proof of the lemma.
\end{proof}

Lemma \ref{SSec32St2LemPF} implies that  $N_{Y^{\prime} / \mathbb{P}^{R-2}} \otimes {\varphi}^{\prime}_{\ast} \mathcal{O}_{X^{\prime}} (P^{\prime}) \cong N_{Y^{\prime} / \mathbb{P}^{R-2}}\oplus N_{Y^{\prime} / \mathbb{P}^{R-2}} (-1)$. This reduces the computation of $h^0 (N_{X^{\prime} / \mathbb{P}^{R-1}})$ to the calculation of $h^0 (N_{Y^{\prime} / \mathbb{P}^{R-2}})$ and $h^0 (N_{Y^{\prime} / \mathbb{P}^{R-2}} (-1))$. It is well known that the Hilbert scheme of curves is generically smooth, irreducible and has the expected dimension when the degree is at least twice the genus, see \cite{CS1987} or \cite{Har1982}. In particular, $\mathcal{I}_{g-2\gamma, \gamma, g-3\gamma}$ is irreducible and generically smooth, of dimension $\lambda_{g-2\gamma, \gamma, R-2} = (g-2\gamma)(R-1)-(R-5)(\gamma-1)$. By the construction in Step I, $D^{\prime} \cong \Gamma$ is general in $\mathcal{M}_{\gamma}$ and the line bundle $\mathcal{O}_{E^{\prime}} \in \Pic^{g-2\gamma} (\Gamma)$ inducing the embedding $Y^{\prime} \subset \mathbb{P}^{R-2}$ is also general. Therefore, $[Y^{\prime}]$ is a general point of $\mathcal{I}_{g-2\gamma, \gamma, g-3\gamma}$ and
\[
    h^0 (Y^{\prime}, N_{Y^{\prime} / \mathbb{P}^{R-2}})
    = (g-2\gamma)(R-1) - (R-5)(\gamma - 1) \, .
\]
Using again the generality of $Y^{\prime} \cong \Gamma \in \mathcal{M}_{\gamma}$ and $\deg Y^{\prime} = g-2\gamma \geq 4\gamma + 5$ with $\gamma\geq 7$, it follows from \cite[Corollary 1.7 $(i)$]{Lop96} that  $h^0 (Y^{\prime}, N_{Y^{\prime} / \mathbb{P}^{R-2}} (-1)) = R-1$, hence
\[
\begin{aligned}
 h^0 (X^{\prime}, {{\varphi}^{\prime}}^{\ast} N_{Y^{\prime} / \mathbb{P}^{R-2}} \otimes \mathcal{O}_{X^{\prime}} (P^{\prime}))
 & = h^0 (Y^{\prime}, N_{Y^{\prime} / \mathbb{P}^{R-2}}) + h^0 (Y^{\prime}, N_{Y^{\prime} / \mathbb{P}^{R-2}}(-1)) \\
 & = R^2 - g + 7\gamma - 6 \, .
\end{aligned}
\]
Now, expressions {\rm (\ref{SSec32Eqh0NXP})} and {\rm (\ref{SSec32Eqh0NYP})} give
\[
 \begin{aligned}
  h^0 (X^{\prime}, N_{X^{\prime} / \mathbb{P}^{R-1}})
  & = h^0 (Y^{\prime}, N_{Y^{\prime} / \mathbb{P}^{R-2}})
  + h^0 (X^{\prime}, {{\varphi}^{\prime}}^{\ast} N_{Y^{\prime} / \mathbb{P}^{R-2}} \otimes \mathcal{O}_{X^{\prime}} (P^{\prime})) \\
  & = R^2 + 2g - \gamma \, .
 \end{aligned}
\]
Therefore,
\begin{equation}\label{SecDimTDPrime}
 \dim T_{[X^{\prime}]} \mathcal{D}^{\prime}_{d-1, g, R-1} = R^2 + 2g - \gamma \, .
\end{equation}
This completes {\rm Step II}.

\medskip

\subsection*{Step III: Non-reduced component \texorpdfstring{$\mathcal{D}^{\prime}_{d-1, g, R-1}$}{D' d-1, g, R-1}}

 Recall that for the dimension of the family $\mathcal{F}^{\prime}$ we found $\dim \mathcal{F}^{\prime} = R^2 + 2g - \gamma - 1 = h^0 (X^{\prime}, N_{X^{\prime} / \mathbb{P}^{R-1}}) - 1$. Together with equation {\rm (\ref{SecDimTDPrime})} this implies that either $\mathcal{D}^{\prime}_{d-1, g, R-1}$ is contained properly in a generically smooth component of dimension $R^2 + 2g - \gamma$ or $\mathcal{D}^{\prime}_{d-1, g, R-1}$ is a non-reduced component. To deduce that $\mathcal{D}^{\prime}_{d-1, g, R-1}$ is an irreducible component of $\mathcal{I}_{d-1, g, R-1}$ it is sufficient to show that $\mathcal{D}^{\prime}_{d-1, g, R-1}$ is not properly contained in a component of $\mathcal{I}_{d-1, g, R-1}$. For this assume the opposite, that is, suppose that there exist an irreducible component $\mathcal{E} \subset \mathcal{I}_{d-1, g, R-1}$ that contains $\mathcal{D}^{\prime}_{d-1, g, R-1}$ properly. This implies that $\dim \mathcal{E} \geq \dim \mathcal{D}^{\prime}_{d-1, g, R-1} + 1 = R^2 + 2g - \gamma$. Also, there must exist a flat family $\mathcal{X}$ of curves of genus $g$ and degree $d-1$ in $\mathbb{P}^{R-1}$ parametrized by  an irreducible curve $T$

\[
\begin{tikzcd}
    \mathcal{X} \arrow[r, hook] \arrow[d] & \mathbb{P}^{R-1} \times T \\
        T &
\end{tikzcd}
\]

\noindent such that $\mathcal{X}_{t_0}$ is a curve from the family $\mathcal{F}^{\prime}$ for some $t_0 \in T$ while $\mathcal{X}_t$ corresponds to a general point of $\mathcal{E}$ for $t \in T \setminus \{t_0\}$. By upper semicontinuity we get
$ \dim T_{[\mathcal{X}_{t}]} \mathcal{E} \leq  \dim T_{[\mathcal{X}_{t_0}]} \mathcal{E} = h^0 (\mathcal{X}_{t_0}, N_{\mathcal{X}_{t_0} / \mathbb{P}^{R-1}}) = R^2 + 2g - \gamma $, whence  $h^0 (\mathcal{X}_{t}, N_{\mathcal{X}_{t} / \mathbb{P}^{R-1}}) = R^2 + 2g - \gamma$ since
$\dim \mathcal{E} \geq \dim \mathcal{D}^{\prime}_{d-1, g, R-1} + 1 = R^2 + 2g - \gamma$. In particular, $\mathcal{E}$ must be generically smooth of dimension $\dim \mathcal{E} = R^2 + 2g - \gamma$. We will show that this is impossible (in fact, further in the proof we will argue only against $\dim \mathcal{E} \geq R^2 + 2g - \gamma$).

The essence of the arguments that follow is that the gonality of the curves from the family $\mathcal{F}^{\prime}$ is too large for the family $\mathcal{F}^{\prime}$ to be contained properly in another family. A general curve $X^{\prime}$ from $\mathcal{F}^{\prime}$ is double cover of a general (in moduli sense) curve $Y^{\prime}$ of genus $\gamma$, hence $\gon (Y^{\prime}) = \lfloor \frac{\gamma + 3}{2} \rfloor$. The Castelnuovo-Severi inequality then implies that
\[
    \gon (X^{\prime}) = 2 \lfloor \frac{\gamma + 3}{2} \rfloor \, .
\]
We can assume that $\mathcal{X}_{t_0}$ is a general curve from $\mathcal{F}^{\prime}$. Since $\mathcal{X}_{t_0} \subset \mathbb{P}^{R-1}$ is linearly normal, it follows by upper semicontinuity that $h^{0} (\mathcal{X}_{t}, \mathcal{O}_{\mathcal{X}_{t}} (1)) = R$. By the Riemann-Roch theorem we obtain that $h^{0} (\mathcal{X}_{t}, \omega_{\mathcal{X}_{t}} (-1)) = \gamma$. Let $B_t$ be the base locus of the linear series $|\omega_{\mathcal{X}_{t}} (-1)|$, and let $b_t := \deg B_t$. Then the linear series $|\omega_{\mathcal{X}_{t}} (-1)(-B_t)|$ must be complete of dimension $\gamma - 1$ and degree $4\gamma - 3 - b_t$. We denote $|\omega_{\mathcal{X}_{t}} (-1)(-B_t)| =: g^{\gamma - 1}_{4\gamma - 3 - b_t}$. Note that in such a case the Riemann-Roch theorem gives that $|\mathcal{O}_{\mathcal{X}_{t}} (1)(B_t)|$ is a very ample linear series on $\mathcal{X}_t$, and has dimension $R-1+b_t$ and degree $d-1+b_t$.

In what follows we will use the variety $\mathcal{W}^{r}_{\delta} (\xi)$ of line bundles defined on a family of curves $\xi : \mathcal{C} \to N$ parametrized by a scheme $N$, that is,
\[
\begin{aligned}
 \mathcal{W}^{r}_{\delta} (\xi)
    = \left\lbrace  (\mathcal{C}_{\nu} , L_{\nu}) \, \mid \, L_{\nu} \mbox{ is a line bundle of degree } \delta  \mbox{ on } \mathcal{C}_{\nu}  \mbox{ and } h^0 (\mathcal{C}_{\nu}, L_{\nu}) \geq r + 1 \right\rbrace \, .
\end{aligned}
\]
For details about them refer to \cite[Section 21]{ACGH2} or \cite{CIK17}. Remark that to each irreducible component $\mathcal{W} \subset \mathcal{W}^{r}_{\delta} (\xi)$ corresponds an irreducible component $\widetilde{\mathcal{W}} \subset \mathcal{W}^{\kappa - \delta + r -1}_{2\kappa - 2 - \delta} (\xi)$, where $\mathcal{W}^{\kappa - \delta + r -1}_{2\kappa - 2 - \delta} (\xi)$ is the dual to $\mathcal{W}^{r}_{\delta} (\xi)$ variety in the relative Jacobian and $\kappa$ is the genus of the curves from the family $\xi : \mathcal{C} \to N$. Indeed, $\dim \widetilde{\mathcal{W}} = \dim \mathcal{W}$.

Consider now the component $\mathcal{E} \subset \mathcal{I}_{d-1, g, R-1}$. With a small abuse of notation, we denote the family of curves parametrized by $\mathcal{E}$ again by $\mathcal{X}$. Thus, the fibers of the family $\mathcal{X} \to \mathcal{E}$ are denoted by $\mathcal{X}_t$ as we consider $t \in \mathcal{E}$. Since $h^{0} (\mathcal{X}_{t}, \mathcal{O}_{\mathcal{X}_{t}} (1)) = R$, the component $\mathcal{E}$ must be generically fibered over a component $\mathcal{W}$ of $\mathcal{W}^{R-1}_{d-1}$ as the fibers consist of the automorphisms of $\mathbb{P}^{R-1}$. Let $\widetilde{\mathcal{W}}$ be the dual of $\mathcal{W}$ in the relative Jacobian scheme. Remark that $\widetilde{\mathcal{W}} \subset \mathcal{W}^{\gamma-1}_{4\gamma - 3}$. The general points of $\widetilde{\mathcal{W}}$ are pairs $(\mathcal{X}_{t}, \omega_{\mathcal{X}_{t}} (-1))$. A priori, such $|\omega_{\mathcal{X}_{t}} (-1))|$ on $\mathcal{X}_{t}$ has base locus $B_t$ of degree $b_t$. The function that returns the degree of the base locus of a series in a variety of line bundles on a moving curve is upper semicontinuous. Hence, there is minimal number, say $b$, such that for every $(\mathcal{X}_{t}, \omega_{\mathcal{X}_{t}} (-1)) \in \widetilde{\mathcal{W}}$ the degree of the base locus of $|\omega_{\mathcal{X}_{t}} (-1)|$ is at least $b$ and for a general $(\mathcal{X}_{t}, \omega_{\mathcal{X}_{t}} (-1))$ the degree is $b_t = b$.

Consider the morphism
\[
 \phi_t \, : \, \mathcal{X}_{t} \to \mathcal{Z}_{t}\subset\mathbb{P}^{\gamma - 1} \, .
\]
induced by $g^{\gamma - 1}_{4\gamma - 3 - b}$, where  $\mathcal{Z}_{t} := \phi_t (\mathcal{X}_{t})$ and $g^{\gamma - 1}_{4\gamma - 3 - b} = |\omega_{\mathcal{X}_{t}} (-1)(-B_t)|$ and $(\mathcal{X}_{t}, \omega_{\mathcal{X}_{t}} (-1)) \in \widetilde{\mathcal{W}}$ is general. If $\mathcal{Z}_{t}$ is not smooth we consider the morphism $\tilde{\phi}_t$ that factors through its normalization $j \, : \, \tilde{\mathcal{Z}}_{t} \to \mathcal{Z}_{t}$

\begin{equation}\label{Sec3FactorDiagram}
\begin{tikzcd}
         & \tilde{\mathcal{Z}}_{t} \arrow[d, "j"] \\
    \mathcal{X}_t \arrow[ru, "\tilde{\phi}_t"] \arrow[r, "{\phi}_t"'] & \mathcal{Z}_{t}
\end{tikzcd}
\end{equation}

\noindent Denote by $k$ the degree of $\phi_t$ (notice that both $\phi_t$ and $\tilde{\phi}_t$ are of the same degree). Let $\tau$ be the geometric genus of $\mathcal{Z}_t$, or equivalently, the genus of $\tilde{\mathcal{Z}}_t$. There are two main cases about the morphism $\phi_t$ that we will consider separately:
\begin{itemize}
 \item[{\rm (i)}] $k = 1$, that is, $\phi_t \, : \, \mathcal{X}_t \to \mathcal{Z}_t$ is a birational morphism, and

 \item[{\rm (ii)}] $k \geq 2$, that is, $\phi_t$ is a multi-covering morphism.
\end{itemize}

We begin with the first case which is easier to deal with. In it we get that $\mathcal{X}_{t}$ must be birational to its image $\mathcal{Z}_{t}$, hence $\tau = g$. Recall that by the Castelnuovo genus bound, if $D$ is a smooth integral curve that admits a birationally very ample linear series $g^{r^{\prime}}_{d^{\prime}}$ then for its genus $g(D)$ we have
\[
 g(D) \leq \pi (d^{\prime}, r^{\prime}) := \binom{m}{2} (r^{\prime} - 1) + m \varepsilon \, ,
\]
where $m = \lfloor \frac{d^{\prime} - 1}{r^{\prime} - 1} \rfloor$ and $\varepsilon = (d^{\prime} - 1) - m (r^{\prime} - 1)$, see \cite[p. 116]{ACGH}. Applying it for the curve $\mathcal{X}_{t}$ we get
\[    g \leq \pi (4\gamma - 3 - b, \gamma -1) \leq \pi (4\gamma - 3 , \gamma -1) = 6\gamma + 4  \, ,
\]
which is impossible due to the assumption $g \geq 6\gamma + 5$ in the theorem.

We turn now to case {\rm (ii)}. To avoid confusion, we recall that in regards for the family of $\mathcal{X} \to \mathcal{E}$ we denote by $\mathcal{X}_{t_0}$ a specialization to a curve in $\mathcal{F}^{\prime}$ while $\mathcal{X}_{t}$ stands for a curve corresponding to a general point of $\mathcal{E}$. Recall that the function returning the gonality of a curve in family is lower semicontinuous, that is, it doesn't increase under specialization. As we remarked already, $\gon (\mathcal{X}_{t_0}) = 2 \lfloor \frac{\gamma + 3}{2} \rfloor$, so
\begin{equation}\label{Sec3GonBound}
 \gon (\mathcal{X}_{t}) \geq \gon (\mathcal{X}_{t_0}) \geq 10 \quad  \mbox{ for } \quad \gamma \geq 7 \, .
\end{equation}
The morphism $\phi_t$ in {\rm (\ref{Sec3FactorDiagram})} induces a linear series $g^{\gamma - 1}_{\frac{4\gamma - 3-b}{k}}$ on $\widetilde{\mathcal{Z}}_t$. Obviously, $k \leq 4$ as $k = 4$ if and only if $b=1$ and $\widetilde{\mathcal{Z}}_t$ is rational. However, the last would imply $\gon (\mathcal{X}_{t}) \leq 4$, which contradicts with {\rm (\ref{Sec3GonBound})}. This leaves $k = 2$ or $k = 3$ as the only options.

Before proceeding further, let us recall that sub-locus $\Sigma^k_{g,\tau} \subset \mathcal{M}_g$  of curves that admit a $k$-sheeted covering map to a curve of genus $\tau$ has dimension $\dim \Sigma^k_{g,\tau} = 2g-2 - (2k-3)(\tau - 1)$, see for example \cite{Lan75}. Note also that there is a rational map
\[
 \Phi \, : \, \mathcal{I}_{d-1, g, R-1} \dashrightarrow \mathcal{M}_g
\]
defined on the open subset of points corresponding to smooth integral curves. In particular, if $\mathcal{E}_0 \subset \mathcal{E}$ is the open subset of smooth curves, then $\Phi (\mathcal{E}_0) \subset \mathcal{M}_g$ and
\[
 \dim \Phi (\mathcal{E}_0) \geq  \dim \Phi (\mathcal{D}^{\prime}_{d-1,g,R-1}) = 2g - \gamma - 1 \, .
\]
$\Phi (\mathcal{E}_0)$ is contained in $\Sigma^k_{g,\tau}$ since $\mathcal{X}_t \to \widetilde{\mathcal{Z}}_t$ is a $k$-sheeted cover. Therefore $\dim \Phi (\mathcal{E}_0) \leq \dim \Sigma^k_{g,\tau}$, which implies $2g-2 - (2k-3)(\tau - 1) \geq \dim \Phi (\mathcal{E}_0) \geq  2g - \gamma - 1$, hence
\begin{equation}\label{Sec3DimOfImagInMgIneq}
 \gamma - 1 \geq (2k - 3) (\tau - 1) \, .
\end{equation}
Since $\widetilde{\mathcal{Z}}_t$ supports $g^1_{\lfloor \frac{\tau + 3}{2} \rfloor}$, it follows that $\gon (\mathcal{X}_t) \leq \frac{k}{2} (\tau + 3)$. Combining the last with {\rm (\ref{Sec3GonBound})} we obtain
\begin{equation}\label{Sec3GenusIneq_tau_gamma}
 \gamma + 2 \leq \frac{k}{2} (\tau + 3) \ .
\end{equation}
Together {(\ref{Sec3DimOfImagInMgIneq})} and {(\ref{Sec3GenusIneq_tau_gamma})} imply that
\begin{equation}\label{Sec3Bound_tau}
 \frac{k}{2} (\tau + 3) \geq (2k-3)(\tau - 1) + 3 \ ,
\end{equation}
whence either $k=3$ and $\tau\leq 3$ or $k=2$.

Consider $k = 3$ and $\tau\leq 3$. Then $\widetilde{\mathcal{Z}}_t$ admits a $g^1_3$, and hence ${\mathcal{X}_t}$ carries a $g^1_9$ via $\tilde{\phi}_t$ in {\rm (\ref{Sec3FactorDiagram})}. This is contraction to the equation {\rm (\ref{Sec3GonBound})}.

It remains to consider the case $k = 2$. In this case inequalities {(\ref{Sec3DimOfImagInMgIneq})} and {(\ref{Sec3GenusIneq_tau_gamma})} imply that
\[
 \gamma - 1 \leq \tau \leq \gamma \, .
\]
Let $\tilde{H}$ be a divisor on $\widetilde{\mathcal{Z}}_t$ corresponding to $\mathcal O_{\mathcal{Z}_t}(1)$. Assume that $\tau = \gamma - 1$. Since $\dim |\tilde{H}| = \tau$, $|\tilde{H}|$ is non-special and hence $\dim |\tilde{H} + z|= \tau + 1$ for any point $z \in \widetilde{\mathcal{Z}_t} $. For $x_1, x_2 \in {\mathcal{X}_t}$ with $\phi_t (x_1) = \phi_t (x_2)$ and $\supp (B_t \cap(x_1 + x_2)) = \emptyset$,
\[
 \begin{aligned}
    \dim |\omega_{\mathcal{X}_t}(-1) + x_1 + x_2| &
    \geq\dim|\omega_{\mathcal{X}_t}(-1)(-B_t + x_1 + x_2)| \\
    & =\dim| \omega_{\mathcal{X}_t}(-1)(-B_t)| + 1 \\
    & =\dim| \omega_{\mathcal{X}_t}(-1)|+1 \, ,
 \end{aligned}
\]
since $\phi^{\ast}_t (\mathcal O_{\mathcal{Z}_t}(1)) = \omega_{\mathcal{X}_t}(-1)(-B_t)$.
Hence,
\[
    \dim| \mathcal O_{\mathcal{X}_t}(1)(-z_1-z_2)|
    \geq \dim| \mathcal O_{\mathcal{X}_t}(1)|-1 \, ,
\]
which contradicts with the very ampleness of $\mathcal O_{\mathcal{X}_t}(1)$.

\noindent The only possibility that remains to consider is $\tau = \gamma$. This tells us that
$|\tilde{H}|$ is a linear series ${g}^{\gamma-1}_{2\gamma - 2 - \frac{b-1}{2}}$ on the curve $\widetilde{\mathcal{Z}}_t$ having genus $\gamma$. Hence $b=1$. In sum, it follows that $\Phi (\mathcal{E}_0) \subset \Sigma^2_{g,\gamma}$ and
$
\dim \Phi^{-1} ( \mathcal{X}_{t} ) = R^2
$,
which gives $\dim \mathcal{E} = \dim \mathcal{E}_0 \leq \dim \Sigma^2_{g,\gamma} + R^2 = R^2 + 2g - \gamma - 1$. This is a contradiction. As a consequence, $\mathcal{D}^{\prime}_{d-1, g, R-1}$ is an irreducible component of $\mathcal{I}_{d-1, g, R-1}$. This completes Step III and thus, the proof of Theorem B.

\medskip

\section{Remarks on the relation between Theorem A and Theorem B}\label{Section_4}

In our motivation for Theorem A in \cite[Sec. 2]{CIK21}, we considered a family of ruled surfaces $S = \mathbb{P}(\mathcal{O}_{\Gamma} \oplus \mathcal{O}_{\Gamma}(-E))$ with a natural projection morphism $f: S \to \Gamma$, where $\Gamma$ runs through $\mathcal{M}_{\gamma}$ and $E$ through $\Div^{g-2\gamma+1}(\Gamma)$, where $g \geq 4\gamma - 2$. We then constructed the family $\mathcal{F}$ of curves $X$ of genus $g$ and degree $d = 2g - 4\gamma + 2$ in $\mathbb{P}^R$, $R = g - 3\gamma + 2$, as images of the divisors $C \in |2\Gamma_0 + 2E\mathfrak{f}|$ under the morphism $\Psi: S \to \mathbb{P}^R$ determined by the linear series $|\Gamma_0 + E\mathfrak{f}|$, where $\Gamma_0$ is the section of $S$ corresponding to
\[
    0 \to \mathcal{O}_{\Gamma} \to \mathcal{O}_{\Gamma} \oplus \mathcal{O}_{\Gamma}(-E) \to \mathcal{O}_{\Gamma}(-E) \to 0.
\]
As we have seen, the image $F = \Psi(S) \subset \mathbb{P}^R$ of such $S$ is a cone over a curve $Y \subset \mathbb{P}^{R-1}$ of genus $\gamma$ and degree $g - 2\gamma + 1$, where $Y$ is the image under $\Psi$ of a divisor in $|\Gamma_0 + E\mathfrak{f}|$. However, in the proof of Theorem A, we opted to exhibit the curves from the family $\mathcal{F}$ as the intersection of cones over the curves parametrized by $\mathcal{I}_{g-2\gamma+1, \gamma, g-3\gamma+1}$ with hypersurfaces of degree 2 in $\mathbb{P}^{g-3\gamma+2}$, as the geometric arguments were more comprehensive and provided a deeper insight into the problem. In our motivation in Section \ref{Section_2}, we showed the geometric relation between the curves $X$ from the family $\mathcal{F}$ and the curves $X^{\prime}$ from the family $\mathcal{F}^{\prime}$. In what follows we will describe the relation between $\mathcal{F}$ and $\mathcal{F}^{\prime}$ using ruled surfaces and elementary transformations between them.

Recall that in the construction of $\mathcal{F}^{\prime}$, we also used the curves $\Gamma \in \mathcal{M}_{\gamma}$. We further considered $E^{\prime}$ running through the set of effective divisors of degree $g-2\gamma$ on $\Gamma$, where $g > 6\gamma + 5$, and then we constructed the ruled surface $S^{\prime} = \mathbb{P}(\mathcal{O}_{\Gamma} \oplus \mathcal{O}_{\Gamma}(-E^{\prime}))$. For any point $q \in \Gamma$, we have $\deg(E - q) \geq 6\gamma + 5$; hence, there exists an effective divisor $E^{\prime} \in \Div^{g-2\gamma}(\Gamma)$ such that $E - q \sim E^{\prime}$. It is not difficult to see that for a point $x \in f^{-1}(q)$, the elementary transformation $\mathrm{elm}_x S$ of $S$ gives the ruled surface $S^{\prime}$. The converse is also true in a sense: for $x^{\prime} = \Gamma^{\prime}_0 \cap q \mathfrak{f}^{\prime}$, the elementary transformation $\mathrm{elm}_{x^{\prime}} S^{\prime}$ gives $S$. The following proposition shows explicitly this relation as it also includes the curves $C$ and $C^{\prime}$ in Figure \ref{Figure_2}.

\begin{prop}\label{Sec4:PropRSElmT}
Let $\Gamma$ be a smooth curve of genus $\gamma \geq 4$ and let $e > 2\gamma$ be an integer. For a divisor $E \in \Div^{e}(\Gamma)$, denote by $S$ the ruled surface $\mathbb{P}(\mathcal{O}_{\Gamma} \oplus \mathcal{O}_{\Gamma}(-E))$ with projective morphism $f: S \to \Gamma$. Similarly, for $E^{\prime} \in \Div^{e-1}(\Gamma)$, denote by $S^{\prime}$ the ruled surface $\mathbb{P}(\mathcal{O}_{\Gamma} \oplus \mathcal{O}_{\Gamma}(-E^{\prime}))$ with projective morphism $f^{\prime}: S^{\prime} \to \Gamma$. Let $\Gamma_0$ be the section of $S$ corresponding to $\mathcal{O}_{\Gamma} \oplus \mathcal{O}_{\Gamma}(-E) \twoheadrightarrow \mathcal{O}_{\Gamma}(-E)$, and let $\Gamma^{\prime}_0$ be the section of $S^{\prime}$ corresponding to $\mathcal{O}_{\Gamma} \oplus \mathcal{O}_{\Gamma}(-E^{\prime}) \twoheadrightarrow \mathcal{O}_{\Gamma}(-E^{\prime})$. Suppose that $q \in \Gamma$ is such that $E \sim E^{\prime} + q$. Then:
\begin{enumerate}[label=(\alph*), leftmargin=*, font=\rmfamily]
    \item For any $x \in f^{-1}(q) \setminus q\mathfrak{f} \cap \Gamma_0$, the elementary transformation of $S$ is $\elm_x S = S^{\prime}$ as $\elm_x \Gamma_0 = \Gamma^{\prime}_0$;

    \item For a curve $C \in |2\Gamma_0 + 2E \mathfrak{f}|$ and $x \in C \cap q\mathfrak{f}$, the proper transform $\tilde{C}$ of $C$ under $\elm_x$ belongs to $|2\Gamma^{\prime}_0 + (2E^{\prime} + q) \mathfrak{f}^{\prime}|$;

    \item For $x^{\prime} = q \mathfrak{f}^{\prime} \cap \Gamma^{\prime}_0$, the elementary transformation of $S^{\prime}$ is $\elm_{x^{\prime}} S^{\prime} = S$  as $\elm_{x^{\prime}} \Gamma^{\prime}_0 = \Gamma_0$;

    \item For a smooth curve $C^{\prime} \in |2\Gamma^{\prime}_0 + (2E^{\prime}+q) \mathfrak{f}|$ and for $x^{\prime}$ as in (c), the proper transform $\tilde{C}^{\prime}$ of $C^{\prime}$ under $\elm_{x^{\prime}}$ belongs to $|2\Gamma_0 + 2E \mathfrak{f}|$.
\end{enumerate}
\end{prop}
\begin{proof}
Recall that the elementary transformation $\elm_x$ blows-up $S$ at $x$, and then contracts the fiber $q \mathfrak{f}$ as $x \in f^{-1}$. Since $x \notin \Gamma_0$ and $e > 2\gamma$, in particular $q$ is not a base point of $|\mathcal{O}_{\Gamma}(E)|$, it follows by \cite[Theorem 48.3]{FP05} that $\elm_x S = \mathbb{P} (\mathcal{O}_{\Gamma} \oplus \mathcal{O}_{\Gamma}(-E+q))$. Since $-E^{\prime} \sim -E + q$, we have $\mathcal{O}_{\Gamma}(-E+q) \cong \mathcal{O}_{\Gamma}(-E^{\prime})$. Therefore, $S^{\prime} = \elm_x S$, which proves (a), as \cite[Theorem 48.3]{FP05} also gives the claim $\elm_x \Gamma_0 = \Gamma^{\prime}_0$. To see (b), notice that $C \sim 2\Gamma_0 + 2E \mathfrak{f}$ and $(\Gamma_0 + E \mathfrak{f}) \cdot \Gamma_0 = -e + e = 0$, which means that $x \notin \Gamma_0$. By a similar argument as in (a), we get $\elm_x S = S^{\prime}$, and since $\elm_x$ transforms the uni-secants from the class of $\Gamma_0 + E\mathfrak{f}$ on $S$, except those that contain $x$, into uni-secants in the class of $\Gamma^{\prime}_0 + E\mathfrak{f}^{\prime}$ on $S^{\prime}$, while those containing $x$ are transformed into $\Gamma^{\prime}_0 + (E-q)\mathfrak{f}^{\prime}$. Since we have that $E^{\prime} \sim E-q$, we conclude for the proper transform $\tilde{C} = \elm_x$ of $C$ that $\tilde{C} \sim \Gamma^{\prime}_0 + (2E^{\prime} + q)\mathfrak{f}^{\prime}$. This proves (b).

\begin{figure}[h]
  \centering

\begin{tikzpicture}[scale=0.62]
    \draw [line width=1.75pt, blue] plot [smooth, tension=0.5] coordinates{(-5,0.5) (-3.5,0.5) (-2.5,1) (-2,2) (-2,3) (-3,4) (-4.5,4.5) (-6,4.5) (-7.5,4) (-8,3) (-8, 2) (-7.5,1) (-6.5,0.5) (-5,0.5)};
    \draw [line width=2pt, red] plot [smooth, tension=0.5] coordinates{(-8,4) (-7.5,3.5) (-7,3.5) (-6,4) (-5,4.5) (-4,5) (-2,5)};
    \draw [line width=2pt, lime!90] plot coordinates{(-5,-0.5) (-5,5.5)};
    \draw [line width=3pt, gray!50] plot coordinates{(-6.5,1) (-3.5,1)};
    \draw [fill=orange!50, line width=0pt] (-5,4.5) circle (5.5pt);
    \node[rotate=0, below right] at (-5, 4.5) {{$x$}};
    \node[rotate=0, right] at (-5, -0.5) {{$q\mathfrak{f}$}};
    \node[rotate=0, above] at (-6.5, 1) {{${\Gamma}_0$}};
    \node[rotate=0] at (-8.5, 2.5) {{$C$}};
    \node[rotate=0] at (-7, 3) {{$\Lambda$}};

    \draw [line width=1.75pt, blue] plot [smooth, tension=0.5] coordinates{(0,6) (2,6) (3,6.5) (3.5,7.5) (3.3,8) (3,8.5) (2.5,9) (2,9.5) (1,10) (0,10) (-2, 9.5) (-3,9) (-3.5, 8.5) (-3.5,7) (-3,6.5) (-2,6) (0,6)};
    \draw [line width=2pt, lime!90] plot [smooth, tension=0.5] coordinates{(0,5) (0,8) (0.5,8.5) (1,8.5)};
    \draw [line width=2pt, orange!50] plot coordinates{(0.5,8) (0.5,10.5)};
    \draw [line width=3pt, gray!50] plot coordinates{(-2,6.5) (2,6.5)};
    \draw [line width=2pt, red] plot [smooth, tension=0.5] coordinates{(-3.5,9.5) (-3,9) (-2,8.5) (-1,8.5) (0,9) (1,9) (2,9.5) (3.5,9.5)};
    \node[rotate=0] at (-7.5, 5.5) {{\tiny $S = \mathbb{P}(\mathcal{O}_{\Gamma} \oplus \mathcal{O}_{\Gamma} (-E))$}};

    \draw [line width=1.75pt, blue] plot [smooth, tension=0.5] coordinates{(5,1) (6,1.3) (7,1.5) (7.5,2) (7.5,3) (7,4) (6,4.5) (4,4.5) (2,4) (1,3) (1, 2) (1.5,1) (2,0.65) (3,0.5) (4,0.5) (5,1)};
    \draw [line width=2pt, red] plot [smooth, tension=0.5] coordinates{(1,4) (2,3.5) (3,3.5) (4,4) (5,4) (6,4.5) (7.5,5)};
    \draw [line width=2pt, orange!50] plot coordinates{(5,-0.5) (5,5.5)};
    \draw [line width=3pt, gray!50] plot coordinates{(3,1) (7,1)};
    \draw [fill=lime!90, line width=0pt] (5,1) circle (5.5pt);
    \node[rotate=0] at (7.5, 5.5) {{\tiny $S^{\prime} = \mathbb{P}(\mathcal{O}_{\Gamma} \oplus \mathcal{O}_{\Gamma} (-E^{\prime}))$}};
    \node[rotate=0, below right] at (5, 1) {{$x^{\prime}$}};
    \node[rotate=0, right] at (5, -0.5) {{$q{\mathfrak{f}^{\prime}}$}};
    \node[rotate=0, above] at (3, 1) {{${\Gamma}^{\prime}_0$}};
    \node[rotate=0] at (8, 2.5) {{$C^{\prime}$}};
    \node[rotate=0] at (2.5, 3) {{$\Lambda^{\prime}$}};

    \draw[->, dashed] plot [smooth, tension=0.5] coordinates{(-2,3.5) (-1, 3.8) (0, 3.8) (1,3.5)};
    \draw[->, dashed] plot [smooth, tension=0.5] coordinates{(1,1.5) (0, 1.2) (-1, 1.2) (-2,1.5)};
    \node[rotate=0] at (-0.5, 4.1) {{\scriptsize $\nu = \elm_{x}$}};
    \node[rotate=0] at (-0.5, 0.85) {{\scriptsize $\nu^{-1} = \elm_{x^{\prime}}$}};
    \node[rotate=0] at (-0.5, 10.95) {{\footnotesize $\Bl_{x} S = \Bl_{x^{\prime}} {S}^{\prime}$}};
    \draw[->, line width=1pt] plot [smooth, tension=0.5] coordinates{(-4,8) (-4.5, 7.8)(-4.75, 7.5) (-4.9, 7) (-5, 6.5) (-5,6)};
    \node[rotate=0] at (-5, 8) {{ $\sigma$}};
    \draw[->, line width=1pt] plot [smooth, tension=0.5] coordinates{(4,8) (4.5, 7.8)(4.75, 7.5) (4.9, 7) (5, 6.5) (5,6)};
    \node[rotate=0] at (5, 8) {{ $\varepsilon$}};
\end{tikzpicture}

  \caption{ $S = \mathbb{P} (\mathcal{O}_{\Gamma} \oplus \mathcal{O}_{\Gamma}(-E))$ and $S^{\prime} = \mathbb{P} (\mathcal{O}_{\Gamma} \oplus \mathcal{O}_{\Gamma}(-{E}^{\prime}))$, $E \sim E^{\prime} + q$. }
\label{Figure_2}
\end{figure}

\noindent Similarly, $\elm_{x^{\prime}}$ blows-up $S^{\prime}$ at $x^{\prime}$, and then contracts the fiber $q \mathfrak{f}^{\prime}$. Since $x^{\prime}$ is the intersection point of $\Gamma^{\prime}_0$ and $q\mathfrak{f}^{\prime}$, it follows by \cite[Theorem 48.1]{FP05} that $\elm_{x^{\prime}} S^{\prime} = \mathbb{P} (\mathcal{O}_{\Gamma} \oplus \mathcal{O}_{\Gamma}(-E^{\prime}-q))$. Since $E^{\prime} + q \sim E$, we get $\elm_{x^{\prime}} S^{\prime} = S$. Remark $\elm_{x^{\prime}} \Gamma^{\prime}_0 = \Gamma_0$ also follows by \cite[Theorem 48.1]{FP05}. This completes the proof of (c). Finally, by \cite[Prop. 36]{FP05}, the linear series $|2\Gamma^{\prime}_0 + (2E^{\prime} + q)\mathfrak{f}^{\prime}|$ has a unique base-point at $x^{\prime} = q\mathfrak{f}^{\prime} \cap \Gamma^{\prime}_0$. In particular, $C^{\prime}$ must contain $x^{\prime}$. Since $C^{\prime} \cdot q\mathfrak{f}^{\prime} = 2$, $q\mathfrak{f}^{\prime}$ intersects $C^{\prime}$ in a second point, which is indeed not on $\Gamma^{\prime}_0$ because $C^{\prime} \cdot \Gamma^{\prime}_0 = 1$. This means that the proper transform $\tilde{C}^{\prime}$ of $C^{\prime}$ under $\elm_{x^{\prime}}$ doesn't not meet $\Gamma_0$. Using that $\Pic (S)$ is generated by $\Gamma_0$ and $f^{\ast}\Pic(\Gamma)$, we find that $\tilde{C}^{\prime} \sim 2\Gamma_0 + 2E \mathfrak{f}$. Remark that the transforms under $\elm_{x^{\prime}}$ of all curves in $|2\Gamma^{\prime}_0 + (2E^{\prime} + q)\mathfrak{f}^{\prime}|$ pass through the point in which $\elm_{x^{\prime}}$ collapses the fiber $q\mathfrak{f}^{\prime}$. This completes the proof of the proposition.
\end{proof}

Now we are in the position to describe the relation between the family $\mathcal{F}$ parametrized by the component $\mathcal{D}_{d, g, R}$ of $\mathcal{I}_{d, g, R}$ and the family $\mathcal{F}^{\prime}$ parametrized by the component $\mathcal{D}^{\prime}_{d, g, R}$ of $\mathcal{I}_{d-1, g, R-1}$.

A general curve $X$ from $\mathcal{F}$ is cut on a cone $F$ over $Y$, $[Y] \in \mathcal{D}_{g-2\gamma+1, \gamma, R-1}$ general, by a hypersurface of degree two. Such $X$ can be viewed as the image of a curve $C$ in the linear class $2Y_0 + 2E \mathfrak{f}$ on the ruled surface $S = \mathbb{P} (\mathcal{O}_Y \oplus \mathcal{O}_Y (-E))$ with $f : S \to Y$, under the morphism $\Psi$ determined by $|Y_0 + 2E \mathfrak{f}|$, where $E \in |\mathcal{O}_Y (1)|$ and $Y_0$ is the section of $S$ corresponding to $\mathcal{O}_Y \oplus \mathcal{O}_Y (-E) \twoheadrightarrow \mathcal{O}_Y (-E)$. For a point $x \in C$ we get by Proposition \ref{Sec4:PropRSElmT} the elementary transformation $\elm_x S$ transforms $S$ into a ruled surface $S^{\prime}$ such that:
\begin{itemize}
 \item $S^{\prime} = \mathbb{P} (\mathcal{O}_Y \oplus \mathcal{O}_Y (-E+q))$ has a projection $f^{\prime} : S^{\prime} \to Y$ and a section $Y^{\prime}_0$ corresponding to $\mathcal{O}_Y \oplus \mathcal{O}_Y (-E+q) \twoheadrightarrow \mathcal{O}_Y (-E+q)$ for which $Y^{\prime}_0 = \elm_x (Y_0)$;

 \item the sub-series in $|Y_0 + E \mathfrak{f}|$ of sections vanishing at $x$ are transformed into the linear series $|Y^{\prime}_0 + (E-q) \mathfrak{f}^{\prime}|$ on $S^{\prime}$

 \item for the proper transform $\tilde{C} = \elm_x (C)$ of $C$ we have $\tilde{C} \sim 2Y^{\prime}_0 + (2E-q) \mathfrak{f}^{\prime}$.
\end{itemize}
The morphism $\Psi^{\prime}$, determined by $|Y^{\prime}_0 + (E-q) \mathfrak{f}^{\prime}|$, obviously maps $\tilde{C}$ to a curve from the family $\mathcal{F}^{\prime}$. This corresponds to the geometric picture described in Proposition \ref{Sec2:PropInnerProj}, see also Figure \ref{Figure_3}.

\begin{figure}[h]
  \centering

\begin{tikzcd}
     & & \Bl_{x} S \equiv \Bl_{x^{\prime}} S^{\prime} \ar[dll, bend right=10, "\sigma"'] \ar[drr, bend left=10, "\varepsilon"] & & \\
 S \ar[rrrr, bend left=10, dashrightarrow, "\nu = \elm_{x}"] \ar[d, "\Psi"] & &   & & S^{\prime} \ar[llll, bend left=10, dashrightarrow, "\nu^{-1} = \elm_{x^{\prime}}"] \ar[dd, "{\Psi}^{\prime}"] \\
 \mathbb{P}^{R} \supset F \supset X \phantom{.} \ar[rrrrd, "\pi_{ {\Psi (x)}}"] \ar[d, shift left = 9mm, "\varphi"] & &   & & \\
 \phantom{X.} \mathbb{P}^{R-1} \supset Y & &   & & \phantom{X.} X^{\prime} \subset F^{\prime} \subset \mathbb{P}^{R-1} \ar[d, shift left = -9mm, "{\varphi}^{\prime}"] \\
  & &   & & Y^{\prime} \subset \mathbb{P}^{R-2} \phantom{X.}
\end{tikzcd}

  \caption{ $F^{\prime} = {\Psi}^{\prime} (S^{\prime}) \subset \mathbb{P}^{R-1}$ is the inner projection of $F = \Psi (S) \subset \mathbb{P}^{R}$. }
\label{Figure_3}
\end{figure}

Conversely, every curve $X^{\prime}$ from the family $\mathcal{F}^{\prime}$ is the image of a curve $C^{\prime} \sim 2\Gamma^{\prime} + (2E^{\prime} +q) \mathfrak{f}^{\prime}$ on a ruled surface $S^{\prime} = \mathbb{P} (\mathcal{O}_{\Gamma} \oplus \mathcal{O}_{\Gamma} (-E^{\prime}))$ under $\Psi^{\prime}$ associated with $|\Gamma^{\prime} + E \mathfrak{f}^{\prime}|$, where $f^{\prime} : S^{\prime} \to \Gamma$ is natural projection, $E^{\prime} \in \Div^{g-2\gamma}$, $\Gamma^{\prime}_0$ is the section corresponding to $\mathcal{O}_{\Gamma} \oplus \mathcal{O}_{\Gamma} (-E^{\prime})\twoheadrightarrow \mathcal{O}_{\Gamma} (-E^{\prime})$ and $q$ is a point on $\Gamma$. By Proposition \ref{Sec4:PropRSElmT} the elementary transformation $\elm_{x^{\prime}} S^{\prime}$, $x^{\prime} = \Gamma^{\prime}_0 \cap q\mathfrak{f}^{\prime}$ transforms $S^{\prime}$ into a ruled surface $S$ such that:
\begin{itemize}
 \item $S = \mathbb{P} (\mathcal{O}_{\Gamma} \oplus \mathcal{O}_{\Gamma} (-E^{\prime}-q))$ has a projection $f : S \to \Gamma$ and a section $\Gamma_0$ corresponding to $\mathcal{O}_{\Gamma} \oplus \mathcal{O}_{\Gamma} (-E^{\prime}-q) \twoheadrightarrow \mathcal{O}_{\Gamma} (-E{\prime}-q)$ for which $\Gamma_0 = \elm_{x^{\prime}} (\Gamma^{\prime}_0)$;

 \item the series in $|\Gamma^{\prime}_0 + E^{\prime} \mathfrak{f}^{\prime}|$ transforms into the linear sub-series $|\Gamma_0 + (E^{\prime}+q) \mathfrak{f}|$ on $S$, as it consists of the sections vanishing at the point into which $\elm_{x^{\prime}}$ collapses the fiber $q \mathfrak{f}^{\prime}$;

 \item for the proper transform $\tilde{C}^{\prime} = \elm_{x^{\prime}} (C^{\prime})$ of $C^{\prime}$ we have $\tilde{C}^{\prime} \sim 2\Gamma_0 + (2E^{\prime}+2q) \mathfrak{f}$.
\end{itemize}
The morphism $\Psi$, determined by $|\Gamma_0 + (E^{\prime}+q) \mathfrak{f}|$, maps $\tilde{C}^{\prime}$ to a curve from the family $\mathcal{F}$.

\begin{remark}
Smooth curves lying on a cone and passing through its vertex have been studied by Jaffe, see \cite{Jaf91}, and Catalisano and  Gimigliano, see \cite{CG99}. Using their results we can identify a general curve $X^{\prime}$ from the family $\mathcal{F}^{\prime}$ as \emph{cut on the cone $F^{\prime} \subset \mathbb{P}^{R-1}$, as in the theorem, by a cubic hypersurface that contains $g - 2\gamma - 1$ lines from the ruling of $F^{\prime}$}.
\end{remark}

\end{document}